\documentclass[12pt]{amsart}
\usepackage{amssymb,amsmath,amsthm,amscd,float,pxfonts,rotating,pdflscape}
\usepackage[graphicx]{realboxes}
\usepackage{tikz}
\usetikzlibrary{matrix}

\restylefloat{table}
\newtheorem{theorem}{Theorem}
\newtheorem{lemma}[theorem]{Lemma}
\newtheorem{corollary}[theorem]{Corollary}
\newtheorem{proposition}[theorem]{Proposition}
\newtheorem{definition}[theorem]{Definition}
\newtheorem{fact}[theorem]{Fact}
\theoremstyle{remark}
\newtheorem{remark}{Remark}

\theoremstyle{remark}

\numberwithin{theorem}{section} \numberwithin{equation}{section}

\newcommand{\re}{\textnormal{Re}}
\newcommand{\im}{\textnormal{Im}}

\restylefloat{table}
\newcommand{\nocontentsline}[3]{}
\newcommand{\tocless}[2]{\bgroup\let\addcontentsline=\nocontentsline#1*{#2}\egroup}

\newcommand{\app}[4]{F_{#1}\!
  \left(\left.{#2 \atop #3}\right| #4 \right) }
\newcommand{\tapp}[4]{\widetilde{F}_{#1}\!
  \left(\left.{#2 \atop #3}\right| #4 \right) }
\newcommand{\hpg}[5]{{}_{#1}F_{#2}\!
  \left(\left.{#3 \atop #4}\right| #5 \right) }
\newcommand{\hpgo}[2]{{}_{#1}F_{#2}}

\begin{document}
\title[Special function identities from superelliptic Kummer varieties]{Special function identities \\[0.2em] from superelliptic Kummer varieties}
\author{Adrian Clingher, Charles F. Doran, Andreas Malmendier}

\address{Department of Mathematics and Computer Science, University of Missouri -- St.~Louis, \break \indent St.~Louis, MO 63121}
\email{clinghera@umsl.edu}

\address{Department of Mathematics, University of Alberta, Edmonton, Alberta  T6G 2G1, \break \indent Department of Physics, University of Maryland, College Park, MD 20742} 
\email{charles.doran@ualberta.ca, doran@umd.edu}

\address{Department of Mathematics and Statistics, Utah State University,
Logan, UT 84322}
\email{andreas.malmendier@usu.edu}

\begin{abstract}
We prove that the factorization of Appell's generalized hypergeometric series satisfying the so-called quadric property
into a product of two Gauss' hypergeometric functions has a geometric origin: we first construct
a generalized Kummer variety as minimal nonsingular model for a product-quotient surface with only rational double points from a pair of superelliptic curves 
of genus $2r-1$ with $r \in \mathbb{N}$. We then show that this generalized Kummer variety is equipped with two fibrations with fibers of genus $2r-1$. 
When periods of a holomorphic two-form over carefully crafted transcendental two-cycles on the generalized Kummer variety are evaluated using either of the two fibrations, 
the answer must be independent of the fibration and the aforementioned family of special function identities is obtained.
This family of identities can be seen as a multivariate generalization of Clausen's Formula. Interestingly, this paper's finding bridges
Ernst Kummer's two independent lines of research, algebraic transformations for the Gauss' hypergeometric 
function and nodal surfaces of degree four in $\mathbb{P}^3$.
\end{abstract}

\subjclass[2010]{14D0x, 14J28, 33C65}

\maketitle

\section{Introduction}
The Appell series $F_2$ is a hypergeometric series in two variables that was introduced by Paul Appell in 1880 as generalization of Gauss' hypergeometric series $\hpgo21$ of one variable. 
Appell established the set of linear partial differential equations of rank four which this function is a solution of, and found various reduction formulas for this and three closely related series in terms of hypergeometric series 
of one variable. More recently, Vidunas \cite{MR2514458} derived corresponding relations between univariate specializations of Appell's functions and univariate hypergeometric functions.
In this article, we consider the Appell hypergeometric functions satisfying the quadric property that four of its linearly independent solutions are quadratically related \cite{MR960834}.
This means that Appell's system of linear partial differential equations can be reduced to a pair of ordinary differential equations by a change coordinates and a renormalization.
In fact, an even stronger result is true: every Appell hypergeometric function $F_2$ satisfying the quadric property can be written as a product of two Gauss' hypergeometric functions of one variable.  This result can be regarded as a multivariate generalization of the classical Clausen's Identity \cite{MR1577682} that relates generalized univariate hypergeometric functions of type $\hpgo32$ to products of two Gauss' hypergeometric functions\footnote{There are other ways of generalizing of this identity: 
in \cite{MR2794653} a generalization of Clausen's Identity to \emph{univariate} generalized hypergeometric functions of rank $4$ and $5$ was given.}.
At the same time, this multivariate identity incorporates many other classical identities for Gauss' hypergeometric function. For example, taken in conjunction with an obvious symmetry of the $F_2$-system, the decomposition 
realizes the most important quadratic identity\footnote{The quadratic identity in question consists in fact of two quadratic identities, but for complementary
moduli of the Gauss' hypergeometric function in question.} for the Gauss' hypergeometric function $\hpgo21$ first derived by Kummer \cite{MR1578088}.
A complete list of all quadratic identities was later given by Goursat \cite{MR1508709}.
The importance of these identities lies in the fact that they often enable us to determine the value of a hypergeometric function $\hpgo21(z)$ at $z=-1$
by relating it to the value of another hypergeometric function at $z=1$ which, in turn, can be evaluated using Gauss' theorem \cite{MR1578093}.
There are many generalizations giving the value of a hypergeometric function as an algebraic number at special rational values that are based on higher-degree identities.

The reduction formula -- which we will also refer to as \emph{Multivariate Clausen Identity} -- stems from the fact that a quadric in $\mathbb{P}^3$ is ruled and decomposes into $\mathbb{P}^1 \times \mathbb{P}^1$.
If we consider this ruled quadric the period domain of a family of polarized varieties, then the decomposition formula -- as we will show -- has a purely algebro-geometric interpretation as period computation 
on a generalized Kummer variety. In the simplest case, this generalized Kummer variety is an ordinary Kummer surface of two non-isogenous elliptic curves.
On the one hand, we can view this Kummer surfaces as a resolution of the quotient of an Abelian surface by an involution. On the other hand, the Kummer surface can be equipped 
with the structure of a Jacobian elliptic $K3$ surface of Picard-rank 18 \cite{MR1013073}. However, for the purpose of realizing hypergeometric function identities only those Jacobian elliptic fibrations will prove relevant
that relate the $K3$ surface to an extremal Jacobian rational elliptic surface by a quadratic twist or a quadratic base transformation and that have non-trivial two-torsion. 
Equivalently, we can say that the Kummer surface is equipped with two Jacobian elliptic fibrations, an isotrivial fibration and a non-isotrivial one.

The central idea of this article is that periods of a holomorphic two-form over carefully crafted transcendental two-cycles on a Kummer surface can be evaluated 
using either Jacobian elliptic fibration. As the answer must be independent of the fibration used we obtain a special function identity that is the Multivariate Clausen Identity.
However, the most general version of the Multivariate Clausen Identity which has additional free rational parameters cannot be realized in the framework of a classical Kummer surface.
Therefore, we will construct a generalized Kummer variety as minimal nonsingular model for a product-quotient surface with only rational double points from a pair 
of two highly symmetric curves of genus $2r-1$ with $r \in \mathbb{N}$, so-called superelliptic curves. Furthermore, we will establish on such a generalized Kummer variety the structure 
-- not of elliptic fibrations with section -- but of fibrations with fibers of genus $2r-1$. 
When periods of a holomorphic two-form over carefully crafted transcendental two-cycles on the generalized Kummer variety are evaluated using either of two fibrations, 
the answer must be independent of the fibration and the aforementioned family of special function identities is obtained.
It is interesting that Ernst Kummer was the first person who investigated both quadratic identities for the Gauss' hypergeometric 
function as well as the nodal surfaces of degree $4$ in $\mathbb{P}^3$ \cite{MR1579281} 
that we call Kummer surfaces today. This article connects precisely these two subjects using today's understanding of variations of Hodge structure and period mappings.

This article is structured as follows: in Section~\ref{Appell} we will review some basic facts about Gauss'  hypergeometric functions and the multivariate generalization $F_2$. For Appell's hypergeometric system
satisfying the quadric property we prove that the associated Pfaffian system decomposes as outer tensor product of two rank-two hypergeometric systems. Next, we will review the 
Multivariate Clausen Identity that decomposes the holomorphic solution $F_2$ into a product of two Gauss' hypergeometric functions  $\hpgo21$. In Section~\ref{Sec:SEC} we will prove some basic facts about superelliptic curves, the resolution of their singularities, the construction of a holomorphic one-form and the lifting of certain cyclic group actions to the resolution curves,
and compute period integrals of the first kind in terms of Gauss' hypergeometric functions.  In Section~\ref{Sec:Kummer} we will construct a generalized Kummer variety 
as minimal nonsingular model of a product-quotient surface with only rational double points from a pair of two superelliptic curves, determine its Hodge diamond, and 
the explicit defining equations for certain fibrations with fibers of genus $2r-1$ on it. In Section~\ref{Sec:KummerPeriod}, we prove  that the Multivariate Clausen Identity realizes the equality of
periods of a holomorphic two-form evaluated over a suitable two-cycle using the structure of either of the two constructed fibrations

\tocless\section{Acknowledgments}
\setcounter{section}{1}
The first author acknowledges support from the Simons Foundation through grant no.~208258.
The second-named author acknowledges support from the National Sciences and Engineering Research Council, the Pacific Institute for Mathematical Sciences, 
and a McCalla professorship at the University of Alberta. The third author acknowledges the generous support of the University of Alberta's Faculty of Science Visiting 
Scholar Program.

\bigskip

\section{Appell's hypergeometric function $F_2$}
\label{Appell}
Appell's hypergeometric functions $F_2$ is defined by the following double hypergeometric series
\begin{eqnarray} \label{appf1}
\app2{\alpha;\;\beta_1,\beta_2}{\gamma_1,\gamma_2}{z_1,\,z_2} = \sum_{m=0}^{\infty} \sum_{n=0}^{\infty}
\frac{(\alpha)_{m+n}\,(\beta_1)_m\,(\beta_2)_n}{(\gamma_1)_m\,(\gamma_2)_n\;m!\,n!}\,z_1^m\,z_2^n 
\end{eqnarray}
that converges absolutely for $|z_1|+|z_2| <1$. The series is a bivariate generalization of the Gauss' hypergeometric series
\begin{equation} \label{gausshpg}
\hpg21{\alpha,\,\beta}{\gamma}{z} = \sum_{n=0}^{\infty} 
\frac{(\alpha)_{n}\,(\beta)_n}{(\gamma)_n\,n!}\,z^n
\end{equation}
that converges absolutely for $|z| <1$. 
Outside of the radius of convergence we can define both $\hpgo21$ and $F_2$  by their analytic continuation.
As a multi-valued function of $z$ or $z_1$ and $z_2$, the function $\hpgo21$ or $F_2$, respectively, are 
analytic everywhere except for possible branch loci. For $\hpgo21$, the possible branch points are
located at $z=0$, $z=1$ and $z=\infty$. For $F_2$, the possible branch loci are the union of the following lines
\begin{equation} \label{app2sing}
z_1=0,\quad z_1=1,\quad z_1=\infty, \quad z_2=0,\quad z_2=1,\quad z_2=\infty, \quad z_1+z_2=1.
\end{equation}
We call the branch obtained by introducing 
a cut from $1$ to $\infty$ on the real $z$-axis or the real $z_1$- and $z_2$-axes, respectively,
the principal branch of  $\hpgo21$ and $F_2$, respectively. 
The principal branches of  $\hpgo21$ and $F_2$ are entire functions in $\alpha, \beta$ or $\alpha, \beta_1, \beta_2$,
and meromorphic in $\gamma$ or $\gamma_1, \gamma_2$ with poles for $\gamma, \gamma_1, \gamma_2 =0, -1, -2, \dots$.
Except where indicated otherwise we always use principal branches.

Appell's function $F_2$ satisfy a Fuchsian\footnote{Fuchsian means linear homogeneous and with regular singularities}
system of partial differential equations analogous to the hypergeometric equation for the function $\hpgo21$.
The differential equation satisfied by $\hpgo21$ is
\begin{equation} \label{eq:euler}
z(1-z)\,\frac{d^2F}{dz^2}+
\big(\gamma-(\alpha+\beta+1)\, z\big)\frac{dF}{dz}-\alpha\,\beta\,F=0.
\end{equation}
It is a Fuchsian equation with three regular singularities at $z=0$, $z=1$ and $z=\infty$ with local exponent differences equal to $1-\gamma$, $\gamma-\alpha-\beta$, and $\alpha-\beta$, respectively.
The system of partial differential equations satisfied by $F_2$ is given by
\begin{equation} \label{app2system}
\begin{split}
z_1(1-z_1)\frac{\partial^2F}{\partial z_1^2}-z_1z_2\frac{\partial^2F}{\partial z_1\partial z_2}
+\left(\gamma_1-(\alpha+\beta_1+1)z_1\right)\frac{\partial F}{\partial z_1}-\beta_1z_2\frac{\partial F}{\partial z_2}
-\alpha \beta_1F=0,\\
z_2(1-z_2)\frac{\partial^2F}{\partial z_2^2}-z_1z_2\frac{\partial^2F}{\partial z_1\partial z_2}
+\left(\gamma_2-(\alpha+\beta_2+1)z_2\right)\frac{\partial F}{\partial z_2}-\beta_2  z_1\frac{\partial F}{\partial z_1}
- \alpha \beta_2F=0.
\end{split}
\end{equation}
This is a holonomic system of rank 4 whose singular locus on $\mathbb{P}^1\times\mathbb{P}^1$ is the union of the lines in~(\ref{app2sing}).
For $\re(\gamma)>\re(\beta)>0$, the Gauss' hypergeometric function $\hpgo21$ has an integral representation
\begin{equation}
\hpg21{\alpha,\,\beta}{\gamma}{z} = \frac{\Gamma(\gamma)}{\Gamma(\beta) \, \Gamma(\gamma-\beta)} \, \int_0^1 \frac{dx}{x^{1-\beta} \, (1-x)^{1+\beta-\gamma} \, (1- z\, x)^{\alpha}} \;.
\end{equation}
We have the following well-known generalization for the Appell hypergeometric function:
\begin{lemma}
For $\re{(\gamma_1)} > \re{(\beta_1)} > 0$ and $\re{(\gamma_2)} > \re{(\beta_2)} > 0$, we have the following integral representation for Appell's hypergeometric series
\begin{equation}
\label{IntegralFormula}
\begin{split}
\app2{\alpha;\;\beta_1,\beta_2}{\gamma_1,\gamma_2}{z_1,\,z_2} = \frac{\Gamma(\gamma_1) \, \Gamma(\gamma_2)}{\Gamma(\beta_1) \, \Gamma(\beta_2) \, \Gamma(\gamma_1 - \beta_1) \, \Gamma(\gamma_2-\beta_2)} \quad \qquad\\
\times \, \int_0^1 du \int_0^1 dx \; 
\frac{1}{u^{1-\beta_2} \, (1-u)^{1+\beta_2-\gamma_2} \, x^{1-\beta_1} \, (1-x)^{1+\beta_1-\gamma_1} \, (1-z_1 \, x - z_2 \, u)^{\alpha}} \;.
\end{split}
\end{equation}
\end{lemma}
\begin{proof}
The series expansion is obtained for $|z_1| + |z_2| <1$ as follows
\begin{equation}
\begin{split}
& \int_0^1 du \int_0^1 dx \; 
\frac{1}{u^{1-\beta_2} \, (1-u)^{1+\beta_2-\gamma_2} \, x^{1-\beta_1} \, (1-x)^{1+\beta_1-\gamma_1} \, (1-z_1 \, x - z_2 \, u)^{\alpha}}\\
 = &  \sum_{m,n \ge 0} \dfrac{\left(\alpha\right)_{m+n}}{m! \, n!}\,  \int_0^1  \dfrac{(z_1 \, x)^m \, dx}{x^{1-\beta_1}\, (1-x)^{1+\beta_1-\gamma_1}} \, \int_0^1 \dfrac{(z_2 \, u)^n \, du}{u^{1-\beta_2}\, (1u)^{1+\beta_2-\gamma_2}} \\
 = &  \sum_{m,n \ge 0} \dfrac{\left(\alpha\right)_{m+n}}{m! \, n!}\,  z_1^m \, z_2^n \; \frac{\Gamma\left(\beta_1+m\right) \, \Gamma\left(\gamma_1-\beta_1\right)}{\Gamma(\gamma_1+m)}
 \frac{\Gamma\left(\beta_2+n\right) \, \Gamma\left(\gamma_2-\beta_2\right)}{\Gamma(\gamma_2+n)} \\
 = & \frac{ \Gamma(\beta_1) \,  \Gamma\left(\gamma_1-\beta_1\right) \, \Gamma(\beta_2) \, \Gamma\left(\gamma_2-\beta_2\right)}{\Gamma(\gamma_1) \, \Gamma(\gamma_2)} \sum_{m,p \ge 0} \dfrac{\left(\alpha\right)_{m+p} \, \left(\beta_1\right)_m \, \left(\beta_2\right)_n}{m! \, (\gamma_1)_m \, n! \, (\gamma_2)_n}\,  z_1^m \, z_2^n  \;.
 \end{split}
\end{equation}
This proves the lemma.
\end{proof}
The connection between Gauss' hypergeometric function $\hpgo21$  and Appell's hypergeometric function $F_2$ is given by an integral transform:
\begin{corollary}
\label{EulerIntegralTransform}
For $\re{(\gamma_1)} > \re{(\beta_1)} > 0$ and $\re{(\gamma_2)} > \re{(\beta_2)} > 0$, we have the following integral relation between
Gauss' hypergeometric function $\hpgo21$  and Appell's hypergeometric function $F_2$:
\begin{equation}
\label{IntegralTransform}
\begin{split}
 \frac{1}{A^{\alpha}} \;
 \app2{\alpha;\;\beta_1,\beta_2}{\gamma_1,\gamma_2}{\frac{1}{A}, \, 1 - \frac{B}{A}} = - \frac{\Gamma(\gamma_2) \, (A-B)^{1-\gamma_2}}{ \Gamma(\beta_2)\, \Gamma(\gamma_2-\beta_2)} \quad \\
 \times \;  \int_A^B \frac{dU}{ (A-U)^{1-\beta_2} \, (U-B)^{1+\beta_2-\gamma_2} \, U^{\alpha}} \; \hpg21{\alpha,\,\beta_1}{\gamma_1}{\frac{1}{U}}  \;.
\end{split} 
\end{equation}
\end{corollary}
\begin{proof}
Using the variable $U =(1-z_2 u)/z_1$ or $u=(1-z_1U)/z_2$, we obtain
\begin{equation*}
\begin{split}
&\frac{1}{(1-z_2 \, u)^\alpha} \, \int_0^1 \frac{dx}{x^{1-\beta_1} \, (1-x)^{1+\beta_1-\gamma_1} \, (1- \frac{z_1}{1- z_2 \, u} \, x)^\alpha}\\
= \;&  \frac{\Gamma(\beta_1) \, \Gamma(\gamma_1-\beta_1)}{\Gamma(\gamma_1)} \,  \frac{1}{z_1^\alpha \, U^\alpha} \, \hpg21{\alpha,\,\beta_1}{\gamma_1}{\frac{1}{U}}  \;.
\end{split}
\end{equation*}
Setting $z_1=1/A$ and $z_2=1-B/A$ or, equivalently, $A=1/z_1$ and $B=(1-z_2)/z_1$, we obtain
\begin{equation}
\begin{split}
& \int_0^1 du\;  \frac{1}{u^{1-\beta_2} \, (1-u)^{1+\beta_2-\gamma_2} }  \;  \frac{1}{z_1^{\alpha} \, U^{\alpha}} \, \hpg21{\alpha,\,\beta_1}{\gamma_1}{\frac{1}{U}}  \\
= \, &   - \, \frac{\left(\frac{z_2}{z_1}\right)^{1-\gamma_2}}{z_1^\alpha} \int_A^B \frac{dU}{ (A-U)^{1-\beta_2} \, (U-B)^{1+\beta_2-\gamma_2} \, U^\alpha}  \; \hpg21{\alpha,\,\beta_1}{\gamma_1}{\frac{1}{U}} \;.
 \end{split}
\end{equation}
Equation~(\ref{IntegralTransform}) follows. 
\end{proof}
\begin{remark}
Equation~(\ref{IntegralFormula})  is the two-parameter generalization of the \emph{Euler integral transform} formula relating $\hpgo21$ and $\hpgo32$.
To see this in more detail, let us start with the classical Euler integral transform formula, i.e.,
\begin{equation}
\label{3F2specialization}
\begin{split}
 \hpg32{\alpha,\;\; \beta_1, \;\; 1+\alpha-\gamma_2}{\gamma_1, \, 1+ \alpha-\gamma_2 + \beta_2}{z_1}
 = \frac{\Gamma(\gamma_1)}{\Gamma(\beta_1) \, \Gamma(\gamma_1-\beta_1)} \quad \\
 \times \; \int_0^1 \frac{dx}{x^{1-\beta_1} \, (1-x)^{1+\beta_1-\gamma_1}} \;
  \hpg21{\alpha,\; \;  1+\alpha-\gamma_2}{1+ \alpha-\gamma_2 + \beta_2}{z_1 \, x} \;.
\end{split}
\end{equation}
We compare Equation~(\ref{3F2specialization}) with an expression similar to Equation~(\ref{IntegralFormula}), but where we have replaced the ramification point $x=1$ by $x=\infty$ 
as upper integration limit\footnote{As we will explain later, this corresponds to a different choice of $A$-cycle
in the fiber of a variety fibered over $\mathbb{P}^2$ whose Picard-Fuchs equation is given by Appell's hypergeometric system} and
changed the normalization factor, i.e.,
\begin{equation}
\begin{split}
\tapp2{\alpha;\;\beta_1,\beta_2}{\gamma_1,\gamma_2}{z_1,\,z_2} = \frac{ (-1)^{-\beta_2} \,\Gamma(\gamma_1) \, \Gamma(1+\alpha-\gamma_2) }{\Gamma(\beta_1) \, \Gamma(\beta_2)
 \,\Gamma(\gamma_1 - \beta_1) \,  \Gamma(1+\alpha-\gamma_2+\beta_2) } \quad \qquad\\
\times \, \int_0^\infty du \int_0^1 dx \; 
\frac{1}{u^{1-\beta_2} \, (1-u)^{1+\beta_2-\gamma_2} \, x^{1-\beta_1} \, (1-x)^{1+\beta_1-\gamma_1} \, (1- z_1 \, x -z_2 \, u )^{\alpha}} \;.
\end{split}
\end{equation}
Mapping $u \mapsto 1-1/u$, we obtain
\begin{equation}
\begin{split}
\tapp2{\alpha;\;\beta_1,\beta_2}{\gamma_1,\gamma_2}{z_1,\,z_2} = \frac{\Gamma(\gamma_1)}{\Gamma(\beta_1) \, \Gamma(\gamma_1-\beta_2)} \qquad \quad \\
 \times \; \int_0^1 \frac{dx}{x^{1-\beta_1} \, (1-x)^{1+\beta_1-\gamma_2}} \;
  \hpg21{\alpha,\, 1+\alpha-\gamma_2}{1+ \alpha-\gamma_2 + \beta_2}{z_1 \, x + (z_2-1)} \;,
\end{split}
\end{equation}
and, therefore, 
\begin{equation}
\label{F2tilde}
 \tapp2{\alpha;\;\beta_1,\beta_2}{\gamma_1,\gamma_2}{z_1, \; 1} =  \hpg32{\alpha, \; \; \beta_1, \; \;  1+\alpha-\gamma_2}{\gamma_1, \; \; 1+ \alpha-\gamma_2 + \beta_2}{z_1} \;.
\end{equation}
The functions $F_2$ and $\tilde{F}_2$ satisfy the same system of linear partial differential equations, but different boundary conditions.
This is in agreement with \cite[Thm.~2.1]{MR2514458} where it was shown that the two restrictions
\begin{equation}
\label{2IndepSolns}
 \hpg32{\alpha,\,\beta_1, \, 1+\alpha-\gamma_2}{\gamma_1, \, 1+ \alpha-\gamma_2 + \beta_2}{z_1} \quad \text{and} \quad \app2{\alpha;\;\beta_1,\beta_2}{\gamma_1,\gamma_2}{z_1,1} 
\end{equation}
satisfy the same ordinary differential equation.

\end{remark}
\begin{remark}
Equation~(\ref{IntegralTransform}) is of particular geometric importance since it makes apparent that two linear transformation formulas for Appell's function $F_2$ 
are induced by the maps $(A,B) \mapsto (B,A)$ and $(A,B) \mapsto (1-A,1-B)$ on the variables. In fact, the variables $(A,B)$ will later be identified with geometric moduli.
Concretely, the right hand side of Equation~(\ref{IntegralTransform})  is invariant under interchanging $A$ and $B$ if we map $(\beta_2,\gamma_2) \mapsto (\gamma_2-\beta_2, \gamma_2)$ as well.
This proves the well-known identity 
\begin{equation}
\label{LinearTransfo1}
 \frac{1}{A^{\alpha}} \;
 \app2{\alpha;\;\beta_1,\; \beta_2}{\gamma_1,\gamma_2}{\frac{1}{A}, \; 1 - \frac{B}{A}}
 =  \frac{1}{B^{\alpha}} \;
 \app2{\alpha;\; \beta_1, \; \gamma_2-\beta_2}{\gamma_1,\gamma_2}{\frac{1}{B}, \; 1 - \frac{A}{B}} \;.
 \end{equation}
Similarly, using the classical linear transformation for the hypergeometric function
\begin{equation}
 \hpg21{\alpha,\; \beta_1}{\gamma_1}{\frac{1}{U}} = \left(\frac{U}{U-1} \right)^\alpha \, \hpg21{\alpha,\; \gamma_1 - \beta_1}{\gamma_1}{\frac{1}{1-U}} \;,
\end{equation}
we can check that right hand side of Equation~(\ref{IntegralTransform})  is also invariant under mapping $U \mapsto 1-U$, $(A,B) \mapsto (1-A,1-B)$
if we map $(\beta_1,\gamma_1) \mapsto (\gamma_1-\beta_1, \gamma_1)$ as well, proving the well-known identity 
\begin{equation}
\label{LinearTransfo2}
 \frac{1}{A^{\alpha}} \;
 \app2{\alpha;\;\beta_1,\; \beta_2}{\gamma_1,\; \gamma_2}{\frac{1}{A}, \; 1 - \frac{B}{A}}
 =  \frac{1}{(A-1)^{\alpha}} \;
 \app2{\alpha;\;\gamma_1-\beta_1,\; \beta_2}{\gamma_1, \; \gamma_2}{\frac{1}{1-A}, \; 1 - \frac{1-B}{1-A}} \;.
 \end{equation}
\end{remark}

\subsection{Quadratic relation between solutions}
In \cite{MR960834} Sasaki and Yoshida studied several systems of linear differential equations in two variables holonomic of rank~4.
Using a differential geometric technique they determined in terms of the coefficients of the differential equations the \emph{quadric property} condition that the four linearly independent solutions are quadratically related.
For Appell's hypergeometric system the quadric condition is as follows:
\begin{proposition}[Sasaki, Yoshida]
Appell's hypergeometric system satisfies the quadric property if and only if
\begin{equation}
\label{QuadricProperty}
 \alpha= \beta_1 + \beta_2 - \frac{1}{2}, \; \gamma_1 = 2\beta_1\, \; \gamma_2 = 2\beta_2 \;.
\end{equation}
\end{proposition}
\begin{remark}
For Appell's hypergeometric system satisfying the quadric property the transformations $(A,B) \mapsto (B,A)$ and $(A,B) \mapsto (1-A,1-B)$ are isometries.
This is obvious from Equations (\ref{LinearTransfo1}) and (\ref{LinearTransfo2}).
\end{remark}
In the following, we will be exclusively concerned with Appell's hypergeometric system satisfying this quadric property.
Geometrically, the quadric condition will correspond to one of the Hodge-Riemann relations for a polarized Hodge structure.
The system (\ref{app2system}) of linear differential equations satisfied by the Appell hypergeometric function 
$$ \app2{\beta_1 + \beta_2 - \frac{1}{2};\;\beta_1,\beta_2}{2 \beta_1, \, 2\beta_2}{z_1, \, z_2}$$
can be written as Pfaffian system for the vector-valued function
$$
  \vec{F} = \langle F, \; \theta_{z_1} F,   \; \theta_{z_2} F, \; \theta_{z_1}\theta_{z_2} F \rangle^t \;
$$  
with $\theta_{z_i}= z_i \, \partial_{z_i}$. The Pfaffian system associated with~(\ref{app2system}) is the rank-four system
\begin{equation}
\label{PfaffianSystemF2}
 d\vec{F} = \Omega^{(F_2)} \cdot \vec{F}
\end{equation}
where the connection matrix $\Omega^{(F_2)}$ is given in Equation~(\ref{connectionF2}).
Next, we introduce two copies of the rank-two Pfaffian system associated with the Gauss' hypergeometric function, i.e., for
$$
  \vec{f}_{\Lambda_i} = \langle f(\Lambda_i), \,  \theta_{\Lambda_i} f(\Lambda_i) \rangle^t 
$$ 
where 
$$f(\Lambda_i)= \hpg21{\beta_1 + \beta_2 -\frac{1}{2},\,\beta_2}{\beta_1 + \frac{1}{2}}{\Lambda_i^2} $$
for $i=1, 2$. The two systems have the form
\begin{equation}
\label{PfaffianSystem2F1}
d \vec{f}_{\Lambda_i} = \Omega^{(\,_2F_1)}_{\Lambda_i} \cdot \vec{f}_{\Lambda_i} 
\end{equation}
with connection matrices given in Equation~(\ref{connection2F1}).
The outer tensor product of the two rank-two Pfaffian systems is constructed by introducing $\vec{H} =  \vec{f}_{\Lambda_1} \boxtimes \vec{f}_{\Lambda_2} $, i.e.,
\begin{equation*}
\begin{split}
  \vec{H} = \, \langle f(\Lambda_1) \, f(\Lambda_2), \;  \theta_{\Lambda_1} f(\Lambda_1) \, f(\Lambda_2),  \;  f(\Lambda_1) \, \theta_{\Lambda_2} f(\Lambda_2),  \; \theta_{\Lambda_1} f(\Lambda_1) \, \theta_{\Lambda_2}f(\Lambda_2)\rangle^t \;.
\end{split}
\end{equation*}
 The associated Pfaffian system is the rank-four system
\begin{equation}
\label{PfaffianSystemSqr}
 d\vec{H} = \Omega^{(\,_2F_1 \otimes \,_2F_1)}\cdot \vec{H}
\end{equation}
with the connection form
\begin{equation}
  \Omega^{(\,_2F_1 \otimes \,_2F_1)} =  \Omega^{(\,_2F_1)}_{\Lambda_1} \boxtimes \mathbb{I} +   \mathbb{I}  \boxtimes   \Omega^{(\,_2F_1)}_{\Lambda_2}  \;.
\end{equation}
The connection matrix is given in Equation~(\ref{connectionT}). Conversely, the quadric property implies that the Pfaffian system for 
Appell's hypergeometric system can be decomposed as outer tensor product of two rank-two Fuchsian systems. In particular,
we have the following proposition:
\begin{proposition}
\label{LemmaTensorSystem}
The connection form of Appell's hypergeometric system satisfying the quadric condition decomposes as
\begin{equation}
\label{RelationConnectionMatrices}
   \Omega^{(\,_2F_1)}_{\Lambda_1} \boxtimes \mathbb{I} +   \mathbb{I}  \boxtimes   \Omega^{(\,_2F_1)}_{\Lambda_2}  = g^{-1} \cdot \left.  \Omega^{(F_2)}  \right|_{(z_1,z_2)=T(\Lambda_1,\Lambda_2)} \cdot g + g^{-1}\cdot dg \;
\end{equation}
where the gauge transformation $g$ is given in Equation~(\ref{gauge}) and the transformation $T$ is given by
\begin{equation}
\label{transfo_F2variables}
\begin{split}
z_1&=4\,{
\frac {\Lambda_1\Lambda_2}{ \left( \Lambda_1+\Lambda_2
 \right) ^{2}}} \;,\\
z_2 & =-{\frac { \left( \Lambda^2_1-1 \right)\left( \Lambda^2_2-1 \right)  }{ \left( \Lambda_1+\Lambda_2 \right) ^{2}}}\;.
 \end{split}
\end{equation}
\end{proposition}
\begin{proof}
By direct computation using the given matrices.
\end{proof}
The theorem proves that two Pfaffian systems are equivalent. However, even more is true: one can explicitly relate
certain holomorphic solutions of the two Pfaffian systems directly. To such extent, Vidunas derived an explicit formula in \cite[Eqn.~(35)]{MR2514458}. Here,
we present his formula using the variables $A, B$ of Corollary~\ref{EulerIntegralTransform} and correct a typographic error.
\begin{theorem}[Multivariate Clausen Identity]
\label{thm1}
For $\re{(\beta_1)}, \re{(\beta_2)} > 0$, $|z_1|+|z_2|<1$, $|\Lambda^2_1|<1$, and $|1 - \Lambda^2_2|<1$,
Appell's hypergeometric series factors into two hypergeometric functions according to
\begin{equation}
\label{F2periodc}
\begin{split}
 &  \; \, \app2{\beta_1 + \beta_2 - \frac{1}{2};\;\beta_1, \; \beta_2}{2 \beta_1, \; 2\beta_2}{z_1, \, z_2}  \\
=   \big(\Lambda_1 + \Lambda_2\big)^{2\beta_1+2\beta_2-1}  \; & \hpg21{\beta_1 + \beta_2 -\frac{1}{2},\;\beta_2}{\beta_1 + \frac{1}{2}}{\Lambda_1^2}  \; \hpg21{\beta_1 + \beta_2 -\frac{1}{2},\;\beta_2}{2 \beta_2}{1-\Lambda_2^2} \;
\end{split}
\end{equation}
with
$$
  (z_1, z_2) = \left(\frac { 4\, \Lambda_1\Lambda_2}{ \left( \Lambda_1+\Lambda_2 \right) ^{2}},  -{\frac { \left( \Lambda^2_1-1 \right)\left( \Lambda^2_2-1 \right)  }{ \left( \Lambda_1+\Lambda_2 \right) ^{2}}}\right) \;.
$$ 
\end{theorem}
\begin{proof}
The identity~(\ref{F2periodc}) is equivalent to 
\begin{equation}
\label{F2periodb}
\begin{split}
& \qquad \qquad \quad \frac{1}{A^{\beta_1+\beta_2-\frac{1}{2}}} \,  \app2{\beta_1 + \beta_2 - \frac{1}{2};\;\beta_1,\; \beta_2}{2 \beta_1, \; 2\beta_2}{\frac{1}{A}, \; 1-\frac{B}{A}}\\
= \; & \, 2^{2\beta_1+2\beta_2-1} \, \big(\Lambda_1 \, \Lambda_2\big)^{\beta_1+\beta_2-\frac{1}{2}} \; 
\hpg21{\beta_1 + \beta_2 -\frac{1}{2},\;\beta_2}{\beta_1 + \frac{1}{2}}{\Lambda_1^2}  \; \hpg21{\beta_1 + \beta_2 -\frac{1}{2},\;\beta_2}{2 \beta_2}{1-\Lambda_2^2}   \;,
\end{split}
\end{equation}
where
$$
 (z_1,z_2) = \left(  \frac{1}{A}, 1- \frac{B}{A} \right) \;, \qquad (A, B)  = \left( \frac{1}{z_1}, \frac{1-z_2}{z_1} \right) 
$$
and
$$
 (z_1, z_2) = \left(\frac { 4\, \Lambda_1\Lambda_2}{ \left( \Lambda_1+\Lambda_2 \right) ^{2}},  -{\frac { \left( \Lambda^2_1-1 \right)\left( \Lambda^2_2-1 \right)  }{ \left( \Lambda_1+\Lambda_2 \right) ^{2}}}\right), \;   (A,B) = \left( \frac { \left( \Lambda_1 + \Lambda_2 \right) ^{2}}{4\Lambda_1\Lambda_2}, \frac { \left( \Lambda_1 \, \Lambda_2 + 1\right) ^{2}}{4\Lambda_1\Lambda_2} \right).
$$ 
We already know that both sides of Equation~(\ref{F2periodb}) satisfy the same system of linear differential equations of rank four. 
In fact, the functions
$$
 \hpg21{\beta_1 + \beta_2 -\frac{1}{2},\;\beta_2}{\beta_1 + \frac{1}{2}}{\Lambda_2^2}  \;\; \text{and} \;\; \hpg21{\beta_1 + \beta_2 -\frac{1}{2},\;\beta_2}{2 \beta_2}{1-\Lambda_2^2}
$$
satisfy the same ordinary differential equation but have a different behavior at the ramification points.  
Expanding both sides of Equation~(\ref{F2periodb}) in terms of $\epsilon_1=\Lambda_1$ and $\epsilon_2=1-\Lambda_2$ and 
showing agreement of the first terms proves equality.
\end{proof}
\begin{remark}
Note that the right hand side of Equation~(\ref{F2periodc}) is up to a renormalization factor a special solution to the the (outer) tensor product of two copies of 
solutions to the same ordinary differential equation as the hypergeometric function
$$f(\Lambda_i)= \hpg21{\beta_1 + \beta_2 -\frac{1}{2},\,\beta_2}{\beta_1 + \frac{1}{2}}{\Lambda_i^2} \;. $$
The aforementioned factor is precisely the entry $g_{11}$ of the gauge transformation $g$ from Proposition~\ref{LemmaTensorSystem}. In fact, the entry $g_{11}$ together with the coordinate transformation~(\ref{transfo_F2variables})
completely determines the entire matrix $g$. This shows that Equation~(\ref{F2periodc}) implies Equation~(\ref{RelationConnectionMatrices}).
\end{remark}
\begin{remark}
Different branches for the renormalization factor on the right hand side of Equation~(\ref{F2periodc}) 
are connected by \emph{duality transformations} on the moduli $\Lambda_1$ and $\Lambda_2$ to the particular branch chosen in Theorem~\ref{thm1}. If we write 
\begin{equation}
\label{F2periodd}
\begin{split}
 &  \; \, \app2{\beta_1 + \beta_2 - \frac{1}{2};\;\beta_1, \; \beta_2}{2 \beta_1, \; 2\beta_2}{z_1, \, z_2}  \\
=   h^{2\beta_1+2\beta_2-1}  \; & \hpg21{\beta_1 + \beta_2 -\frac{1}{2},\;\beta_2}{\beta_1 + \frac{1}{2}}{\Lambda_1^2}  \; \hpg21{\beta_1 + \beta_2 -\frac{1}{2},\;\beta_2}{2 \beta_2}{1-\Lambda_2^2} \;,
\end{split}
\end{equation}
instead of Equation~(\ref{F2periodc}), then the dependence of $(A,B)$, $(z_1,z_2)$ on $\Lambda_1$ and $\Lambda_2$ and the matching branches for $g_{11}=h^{2\beta_1+2\beta_2-1}$ are given in Table~\ref{tab:F2param_a}.
\begin{table}[H]
\scalebox{0.7}{
\begin{tabular}{|c|c|c||c|c||c|}
\hline
&&&&& \\[-0.8em]
duality transformation & $A$ & $B$ & $z_1$ & $z_2$ & $h$\\
\hline
\hline
&&&&& \\[-0.8em]
$\begin{array}{c}  (\Lambda_1, \Lambda_2) \mapsto (\pm \Lambda_1^{\pm 1}, \pm \Lambda_2^{\pm 1})  \end{array}$ &
$ \frac { \left( \Lambda_1 + \Lambda_2 \right) ^{2}}{4\Lambda_1\Lambda_2} $ &
$  \frac { \left( \Lambda_1 \, \Lambda_2 + 1\right) ^{2}}{4\Lambda_1\Lambda_2} $ &
$ \frac { 4\, \Lambda_1\Lambda_2}{ \left( \Lambda_1+\Lambda_2 \right) ^{2}}$ &
$ -{\frac { \left( \Lambda^2_1-1 \right)\left( \Lambda^2_2-1 \right)  }{ \left( \Lambda_1+\Lambda_2 \right) ^{2}}}  $&
$ \Lambda_1 + \Lambda_2$\\[0.5em]
\hline
&&&&& \\[-0.8em]
$\begin{array}{c} (\Lambda_1, \Lambda_2) \mapsto (\mp \Lambda_1, \pm \Lambda_2) \\ (A,B) \mapsto (1-A,1-B) \end{array}$ & 
$ - \frac { \left( \Lambda_1 - \Lambda_2 \right) ^{2}}{4\Lambda_1\Lambda_2} $ &
$ - \frac { \left( \Lambda_1 \, \Lambda_2 - 1\right) ^{2}}{4\Lambda_1\Lambda_2} $ &
$ -  \frac {4 \, \Lambda_1\Lambda_2}{ \left( \Lambda_1-\Lambda_2 \right) ^{2}}$ &
$ -{\frac { \left( \Lambda^2_1-1 \right)\left( \Lambda^2_2-1 \right)  }{ \left( \Lambda_1-\Lambda_2 \right) ^{2}}}  $&
$ - \Lambda_1 + \Lambda_2$\\[0.4em]
\hline
&&&&& \\[-0.8em]
$\begin{array}{c}  (\Lambda_1, \Lambda_2) \mapsto (\Lambda_1^{\mp 1}, \Lambda_2^{\pm 1}) \\ (A,B) \mapsto (B,A) \end{array}$ &
$ \frac { \left( \Lambda_1 \Lambda_2 + 1\right) ^{2}}{4\Lambda_1\Lambda_2} $ &
$  \frac { \left( \Lambda_1 + \Lambda_2 \right) ^{2}}{4\Lambda_1\Lambda_2} $ &
$ \frac { 4\, \Lambda_1\Lambda_2}{ \left( \Lambda_1\Lambda_2 + 1 \right) ^{2}}$ &
$ {\frac { \left( \Lambda^2_1-1 \right)\left( \Lambda^2_2-1 \right)  }{ \left( \Lambda_1\Lambda_2 + 1 \right) ^{2}}}  $&
$ \Lambda_1 \Lambda_2 + 1$\\[0.4em]
\hline
&&&&& \\[-0.8em]
$\begin{array}{c}  (\Lambda_1, \Lambda_2) \mapsto ( \mp \Lambda_1^{\mp 1}, \pm \Lambda_2^{\pm 1}) \\ (A,B) \mapsto (1-B,1-A) \end{array}$ &
$ - \frac { \left( \Lambda_1  \Lambda_2 - 1\right) ^{2}}{4\Lambda_1\Lambda_2} $ &
$ - \frac { \left( \Lambda_1 - \Lambda_2 \right) ^{2}}{4\Lambda_1\Lambda_2} $ &
$ -  \frac {4 \, \Lambda_1\Lambda_2}{ \left( \Lambda_1\Lambda_2 - 1\right) ^{2}}$ &
$ {\frac { \left( \Lambda^2_1-1 \right)\left( \Lambda^2_2-1 \right)  }{ \left( \Lambda_1\Lambda_2 - 1\right) ^{2}}}  $&
$ - \Lambda_1 \Lambda_2 + 1$\\[0.4em]
\hline
\end{tabular}}
\caption{Relation between variables and moduli}\label{tab:F2param_a}
\end{table}
\end{remark}
\begin{remark}
\label{alternative}
The left hand side of Equation~(\ref{F2periodc}) has an obvious symmetry given by interchanging $(z_1,z_2) \mapsto (z_2,z_1)$ and simultaneously swapping $(\beta_1,\beta_2) \mapsto (\beta_2,\beta_1)$.
From the integral representation in Equation~(\ref{IntegralFormula}) it is clear that this is equivalent to interchanging the roles of the variables $x$ and $u$.
In terms of the coordinates $(A, B)$ used in Equation~\ref{F2periodb} and $(\Lambda_1, \Lambda_2)$, respectively, interchanging variables amounts to symmetry transformations
\begin{equation}
\label{symmetry}
 (A,B) \mapsto \left(\frac{A}{A-B}, \frac{A-1}{A-B}\right) \;, \qquad  (\Lambda_1,\Lambda_2) \mapsto \left( - \frac{1-\Lambda_2}{1+\Lambda_2}, \frac{1+\Lambda_1}{1-\Lambda_1}\right) \;.
\end{equation}
Invariance of the right hand side of Equation~(\ref{F2periodc})  under transformation~(\ref{symmetry}) using
\begin{equation*}
\begin{split}
 \left( -  \frac{1-\Lambda_2}{1+\Lambda_2} +  \frac{1+\Lambda_1}{1-\Lambda_1}\right)^{2\beta_1+2\beta_2-1} \,  \left(\frac{1 + \Lambda_2}{2}\right)^{2\beta_1+2\beta_2-1}
 \left(1 - \Lambda_1\right)^{2\beta_1+2\beta_2-1} 
 = \, \Big( \Lambda_1 + \Lambda_2\Big)^{2\beta_1+2\beta_2-1} 
 \end{split}
 \end{equation*}
recovers the following two quadratic identities for the Gauss' hypergeometric function that are due to Kummer~\cite[Eq.~(41) and (51)]{MR1578088}:
\begin{equation}
\begin{split}
\hpg21{\beta_1 + \beta_2 -\frac{1}{2},\;\beta_2}{\beta_1 + \frac{1}{2}}{\left( \frac{1-\Lambda_2}{1+\Lambda_2}\right)^2}  =   \left(\frac{1 + \Lambda_2}{2}\right)^{2\beta_1+2\beta_2-1}
\hpg21{\beta_1 + \beta_2 -\frac{1}{2},\;\beta_1}{2 \beta_1}{1-\Lambda_2^2} \;,\\
\hpg21{\beta_1 + \beta_2 -\frac{1}{2},\;\beta_2}{2 \beta_2}{1-\left( \frac{1+\Lambda_1}{1-\Lambda_1}\right)^2}  =    \left(1 - \Lambda_1\right)^{2\beta_1+2\beta_2-1}
\hpg21{\beta_1 + \beta_2 -\frac{1}{2},\;\beta_1}{\beta_2 + \frac{1}{2}}{\Lambda_1^2}  \;.
\end{split}
\end{equation}
\end{remark}

\begin{remark}
The generalized univariate hypergeometric function $\hpgo32$ satisfies the so-called \emph{Clausen Identity} 
\begin{equation}
\label{3F2Clausen_a}
\begin{split}
 \hpg32{\beta_1+\beta_2- \frac{1}{2},\,\beta_1, \, \beta_1-\beta_2+ \frac{1}{2}}{2\beta_1, \, \beta_1 + \frac{1}{2}}{z_1}
 =  \hpg21{\frac{\beta_1}{2} + \frac{\beta_2}{2} - \frac{1}{4}, \, \frac{\beta_1}{2} - \frac{\beta_2}{2} + \frac{1}{4}}{\beta_1 + \frac{1}{2}}{z_1}^2
\end{split}
\end{equation}
if and only if its parameters satisfy the quadric property in Equation~(\ref{QuadricProperty}). 
Using the quadratic transformation formula \cite[Eq.~(15.3.23)]{MR0167642}, the Clausen identity~(\ref{3F2Clausen_a}) simplifies to
\begin{equation}
\label{3F2Clausen_b}
\begin{split}
 \hpg32{\beta_1+\beta_2- \frac{1}{2},\,\beta_1, \, \beta_1-\beta_2+ \frac{1}{2}}{2\beta_1, \, \beta_1 + \frac{1}{2}}{z_1}
 =  (1-\Lambda_1^2)^{2\beta_1+2\beta_2-1} \; \hpg21{\beta_1 + \beta_2 - \frac{1}{2}, \, \beta_2}{\beta_1 + \frac{1}{2}}{\Lambda_1^2}^2
\end{split}
\end{equation}
with $z_1 = -4\Lambda_1^2/(1-\Lambda_1^2)^2$.  Note that the same relation between $z_1$ and $\Lambda_1$ as well as the
the renormalization factor can be obtained by setting $\Lambda_2=\Lambda_1$ in Theorem~\ref{thm1}, i.e.,
\begin{equation}
\label{3F2Clausen_bb}
\begin{split}
& \; \; \app2{\beta_1 + \beta_2 - \frac{1}{2};\;\beta_1, \; \beta_2}{2 \beta_1, \; 2\beta_2}{z_1, 1} \\
 =  \, (1-\Lambda_1^2)^{2\beta_1+2\beta_2-1} \; & \, \hpg21{\beta_1 + \beta_2 -\frac{1}{2},\;\beta_2}{\beta_1 + \frac{1}{2}}{\Lambda_1^2}  \; \
 \hpg21{\beta_1 + \beta_2 -\frac{1}{2},\;\beta_2}{2 \beta_2}{1-\Lambda_1^2}  \;.
 \end{split}
\end{equation}
Therefore, using Equation~(\ref{F2tilde}) we see that the two restrictions
\begin{equation}
 \tapp2{\alpha;\;\beta_1,\beta_2}{\gamma_1,\gamma_2}{z_1, \; 1} \quad \text{and} \quad \app2{\alpha;\;\beta_1,\beta_2}{\gamma_1,\gamma_2}{z_1,1} 
\end{equation}
satisfy theanalogous factorization formulas. By analogy, we call the (unrestricted) Equation~(\ref{F2periodc}) the \emph{Multivariate Clausen Identity}.
\end{remark}

\begin{landscape}
\subsection{Connection matrices}
The connection matrix for the Pfaffian system in Equation~(\ref{PfaffianSystemF2}) is given by
  \begingroup\makeatletter\def\f@size{9}\check@mathfonts
\def\maketag@@@#1{\hbox{\m@th\large\normalfont#1}}%
 \begin{equation}
 \label{connectionF2}
 \Omega^{(F_2)} = \left(
 \begin {array}{cccc} 
 0 & \frac{dz_1}{z_1} & \frac{dz_2}{z_2} & 0 \\
 - \frac {\alpha\,\beta_1 \, dz_1}{z_1-1} & - \frac { \left(  \left( \alpha+\beta_1 \right) \, z_1 - 2\,\beta_1+1 \right) \, dz_1}{z_1\left( z_1-1 \right) }
 	&- \frac {\beta_1 \, dz_1}{z_1-1} & - \frac {dz_1}{z_1-1}+\frac {dz_2}{z_2} \\ 
 - \frac {\alpha\,\beta_2 \, dz_2}{z_2-1} & - \frac {\beta_2 \, dz_2}{z_2-1} 
	&- \frac { \left(  \left( \alpha+\beta_2 \right) \, z_2 - 2\,\beta_2+1 \right) \, dz_2}{z_2\left( z_2-1 \right) } & -\frac{dz_2}{z_2-1} + \frac{dz_1}{z_1} \\
\frac {\alpha\,\beta_1\,\beta_2 \, z_2 \,dz_1}{ \left( z_1-1 \right)  \left( z_1+z_2-1 \right)} + \frac {\alpha\,\beta_1\,\beta_2 \, z_1 \,dz_2}{ \left( z_2-1 \right)  \left( z_1+z_2-1 \right)}
	& \frac {\beta_2 \, \left( \alpha-\beta_1+1 \right) \, z_2 \, dz_1}{ \left( z_1-1 \right)  \left( z_1+z_2-1 \right) } 
	-  \frac {\beta_2\, \left( (\alpha- 2 \beta_1+1) \, \,z_2-\beta_1 \, z_1 -\alpha+2\,\beta_1-1 \right) \, dz_2}{ \left( z_2-1 \right)  \left( z_1+z_2-1 \right) } 
	& \frac {\beta_1 \, \left( \alpha-\beta_2+1 \right) \, z_1 \, dz_2}{ \left( z_2-1 \right)  \left( z_1+z_2-1 \right) } 
	-  \frac {\beta_1\, \left( (\alpha- 2 \beta_2+1) \, \,z_1-\beta_2 \, z_2 -\alpha+2\,\beta_2-1 \right) \, dz_1}{ \left( z_1-1 \right)  \left( z_1+z_2-1 \right) } 
	&  \Omega^{(F_2)}_{44}
	\end {array} 
  \right)
 \end{equation}
 \endgroup
where $\alpha=\beta_1+\beta_2-\frac{1}{2}$ and
  \begingroup\makeatletter\def\f@size{7}\check@mathfonts
\def\maketag@@@#1{\hbox{\m@th\large\normalfont#1}}%
\begin{equation*}
\begin{split}
 \Omega^{(F_2)}_{44} = & - \frac{\left( \left( \alpha+\beta_1-2\,\beta_2+1 \right) z_1^2+
 \left(  \left( 2\,\beta_1-\beta_2-1 \right) z_2-\alpha-3\,\beta_1+2\,\beta_2 \right) z_1+ \left( -2\,\beta_1+1
 \right) z_2+2\,\beta_1-1 \right) \, dz_1}{z_1 \, \left( z_1-1 \right)  \left( z_1+z_2-1 \right)} 
  - \frac{\left( \left( \alpha+\beta_2-2\,\beta_1+1 \right) z_2^2+
 \left(  \left( 2\,\beta_2-\beta_1-1 \right) z_1-\alpha-3\,\beta_2+2\,\beta_1 \right) z_2+ \left( -2\,\beta_2+1
 \right) z_1+2\,\beta_2-1 \right) \, dz_2}{z_2 \, \left( z_2-1 \right)  \left( z_1+z_2-1 \right)}  \;.
 \end{split}
\end{equation*}
\endgroup
 \medskip
 \noindent
The connection matrix for the Pfaffian system  in Equation~(\ref{PfaffianSystem2F1}) is given by
  \begingroup\makeatletter\def\f@size{9}\check@mathfonts
\def\maketag@@@#1{\hbox{\m@th\large\normalfont#1}}%
 \begin{equation}
 \label{connection2F1}
\Omega^{(\,_2F_1)}_{\Lambda_i} = \left(
\begin {array}{cc} 
 0& \frac {d\Lambda_i}{\Lambda_i} \\ 
 - \frac{2 \, \left(2 \, \beta_1 + 2\beta_2 -1\right) \, \beta_2 \, \Lambda_i \, d\Lambda_i}{\Lambda_i^2-1} 
 	& - \frac{\left( (2 \, \beta_1 + 4 \, \beta_2 -1) \, \Lambda_i^2 - 2 \, \beta_1 +1 \right) \, d\Lambda_i}{\Lambda_i \, \left( \Lambda_i^2-1\right)} \
\end {array}
   \right)\;.
 \end{equation}
 \endgroup
The connection matrix for the outer tensor product of two copies of the above Pfaffian system is given by
 \begingroup\makeatletter\def\f@size{9}\check@mathfonts
\def\maketag@@@#1{\hbox{\m@th\large\normalfont#1}}%
 \begin{equation}
 \label{connectionT}
  \Omega^{(\,_2F_1 \otimes \,_2F_1)} = \left( 
 \begin {array}{cccc} 0 & \frac{d \Lambda_1}{\Lambda_1} & \frac {d \Lambda_2}{\Lambda_2} & 0 \\
  - \frac{2 \, \left(2 \, \beta_1 + 2\beta_2 -1\right) \, \beta_2 \, \Lambda_1 \, d\Lambda_1}{\Lambda_1^2-1} 
  	&  - \frac{\left( (2 \, \beta_1 + 4 \, \beta_2 -1) \, \Lambda_1^2 - 2 \, \beta_1 +1 \right) \, d\Lambda_1}{\Lambda_1 \, \left( \Lambda_1^2-1\right)} & 0 &  \frac {d \Lambda_2}{\Lambda_2} \\
  - \frac{2 \, \left(2 \, \beta_1 + 2\beta_2 -1\right) \, \beta_2 \, \Lambda_2 \, d\Lambda_2}{\Lambda_2^2-1}  & 0 
  	&  - \frac{\left( (2 \, \beta_1 + 4 \, \beta_2 -1) \, \Lambda_2^2 - 2 \, \beta_1 +1 \right) \, d\Lambda_2}{\Lambda_2 \, \left( \Lambda_2^2-1\right)} & \frac {d \Lambda_1}{\Lambda_1} \\
  0 &  - \frac{2 \, \left(2 \, \beta_1 + 2\beta_2 -1\right) \, \beta_2 \, \Lambda_2 \, d\Lambda_2}{\Lambda_2^2-1}  & \frac{2 \, \left(2 \, \beta_1 + 2\beta_2 -1\right) \, \beta_2 \, \Lambda_1 \, d\Lambda_1}{\Lambda_1^2-1} 
  	&  - \frac{\left( (2 \, \beta_1 + 4 \, \beta_2 -1) \, \Lambda_1^2 - 2 \, \beta_1 +1 \right) \, d\Lambda_1}{\Lambda_1 \, \left( \Lambda_1^2-1\right)} 
	 - \frac{\left( (2 \, \beta_1 + 4 \, \beta_2 -1) \, \Lambda_2^2 - 2 \, \beta_1 +1 \right) \, d\Lambda_2}{\Lambda_2 \, \left( \Lambda_2^2-1\right)}
	 \end {array}
   \right) \;.
 \end{equation}
 \endgroup 
 \medskip
 \noindent
The gauge transformation relating connection matrices~(\ref{connectionF2}) and (\ref{connectionT}) is given by
  \begingroup\makeatletter\def\f@size{9}\check@mathfonts
\def\maketag@@@#1{\hbox{\m@th\large\normalfont#1}}%
 \begin{equation}
 \label{gauge}
 g = \left( \begin {array}{cccc}  
 \left( \Lambda_1+\Lambda_2 \right) ^{2\,\alpha}&0&0&0\\ 
 - \frac { 2 \, \alpha \,\left( \Lambda_1+\Lambda_2 \right)^{2\,\alpha} \,\Lambda_1 \, \Lambda_2}{\Lambda_1\Lambda_2-1}
 	& - \frac { \Lambda_2 \, \left( \Lambda_1^2 -1\right)\,   \left( \Lambda_1+\Lambda_2 \right) ^{2\,\alpha}}{ \left( \Lambda_1\Lambda_2-1 \right)  \left( \Lambda_1-\Lambda_2 \right) }
 	& \frac { \Lambda_1 \, \left( \Lambda_2^2 -1\right)\,   \left( \Lambda_1+\Lambda_2 \right) ^{2\,\alpha}}{ \left( \Lambda_1\Lambda_2-1 \right)  \left( \Lambda_1-\Lambda_2 \right) } & 0 \\ 
 \frac {  \alpha \,\left( \Lambda_1+\Lambda_2 \right)^{2\,\alpha} \,\left(\Lambda_1^2-1\right) \, \left(\Lambda_2^2-1\right)}{\Lambda^2_1\Lambda^2_2-1}	
	& \frac {\left( \Lambda_1^2 -1\right)\,  \left( \Lambda_2^2 -1\right)\, \left( \Lambda_1+\Lambda_2 \right) ^{2\,\alpha}}{ 2 \, \left( \Lambda^2_1\Lambda^2_2-1 \right) }
	& \frac {\left( \Lambda_1^2 -1\right)\,  \left( \Lambda_2^2 -1\right)\, \left( \Lambda_1+\Lambda_2 \right) ^{2\,\alpha}}{ 2 \, \left( \Lambda^2_1\Lambda^2_2-1 \right) } & 0 \\
- \frac{ \alpha \, \Lambda_1 \, \Lambda_2 \, \left(\Lambda_1^2-1\right) \, \left(\Lambda_2^2-1\right) \, \left(\Lambda_1 + \Lambda_2\right)^{2\alpha} \, 
\left( (2 \, \beta_1-1) \, \Lambda_1 \, \Lambda_2 - 2 \, \beta_1 -1\right)}{\left(\Lambda_1 \, \Lambda_2 +1\right) \, \left(\Lambda_1 \, \Lambda_2 -1\right)^3}
 & g_{42} & g_{43} & - \frac{\left(\Lambda_1+\Lambda_2\right)^{2 \, \alpha} \, \left(\Lambda_1^2-1\right) \, \left(\Lambda_2^2-1\right)}{2 \, \left(\Lambda_1 \, \Lambda_2 -1\right)^2} \end {array} 
    \right)
 \end{equation}
 \endgroup 
with $\alpha=\beta_1+\beta_2-\frac{1}{2}$ and
 \begingroup\makeatletter\def\f@size{6}\check@mathfonts
\def\maketag@@@#1{\hbox{\m@th\large\normalfont#1}}%
\begin{equation*}
\begin{split}
g_{42}   =  \frac {\Lambda_2 \,  \left( \Lambda_1^2-1 \right)  \,  \left( \Lambda_2^2-1 \right) \,  \left( \Lambda_1+\Lambda_2 \right) ^{2\,\alpha} \,
 \left( 2 \, \alpha \, \Lambda_1^2 \, \Lambda_2^2 - \left( (2 \, \beta_1 -1) \, \Lambda_1^2 + 2\, \beta_1 +1 \right) \, \Lambda_1 \, \Lambda_2 + (2 \, \beta_1 +1) \, \Lambda_1^2 - 2 \, \beta_2 \right)}
 { 2 \, \left( \Lambda_1 \, \Lambda_2-1 \right)^{3} \,   \left( \Lambda_1 \, \Lambda_2+1 \right) \,  \left( \Lambda_1-\Lambda_2 \right) }\;, \quad
g_{43}   = - \frac {\Lambda_1 \,  \left( \Lambda_1^2-1 \right)  \,  \left( \Lambda_2^2-1 \right) \,  \left( \Lambda_1+\Lambda_2 \right) ^{2\,\alpha} \,
 \left( 2 \, \alpha \, \Lambda_1^2 \, \Lambda_2^2 - \left( (2 \, \beta_1 -1) \, \Lambda_2^2 + 2\, \beta_1 +1 \right) \, \Lambda_1 \, \Lambda_2 + (2 \, \beta_1 +1) \, \Lambda_2^2 - 2 \, \beta_2 \right)}
 { 2 \, \left( \Lambda_1 \, \Lambda_2-1 \right)^{3} \,   \left( \Lambda_1 \, \Lambda_2+1 \right) \,  \left( \Lambda_1-\Lambda_2 \right) }\;.
 \end{split}
\end{equation*}
\endgroup

\end{landscape}

\section{Superelliptic curves and their periods}
\label{Sec:SEC}
In this section we discuss some basic facts about superelliptic curves and their periods.
Superelliptic curves and Jacobians of curves with superelliptic components have been studied extensively from an arithmetic point 
of view (cf.~\cite{MR1863009, Berger:2015aa, MR3289629}).
However, our main focus will be on the computation of periods of the first kind for such superelliptic curves.

\subsection{A generalization of the Legendre normal form}
Every elliptic curve, i.e.,  smooth projective algebraic curve of genus one, can be written in Legendre normal form as
\begin{equation}
\label{LegendreNormalForm}
 y^2 = x \; (x-1) \; (x-\lambda) 
\end{equation}
with $\lambda \not \in \lbrace 0, 1, \infty \rbrace$.  

\begin{definition}
For four positive integers $r, s, p, q \in \mathbb{N}$ with $2s-p>0$, $p+q-s>0$, $q-r+s>0$ let 
$SE(\lambda)^{2r}_{s,p,q}$ be the plane algebraic curve given in projective coordinates $[X:Y:Z]\in \mathbb{P}^2$ by
\begin{equation}
\label{SuperLegendre1proj}
SE(\lambda)^{2r}_{s,p,q}: \quad Y^{2r} \, Z^{3s-2r-p+q}= X^{p+q-s} \; (X-Z)^{2s-p} \; (X-\lambda \, Z)^{2s-p} \;.
\end{equation}
\end{definition}
\begin{remark}
Unless necessary we will not distinguish between the plane algebraic curve in Equation~(\ref{SuperLegendre1proj})
and the affine curve given by
\begin{equation}
\label{SuperLegendre1}
SE(\lambda)^{2r}_{s,p,q}: \quad y^{2r} = x^{p+q-s} \; (x-1)^{2s-p} \; (x-\lambda)^{2s-p} \;.
\end{equation}
\end{remark}
To generalize the Legendre normal form in Equation~(\ref{LegendreNormalForm})
we choose three positive integers $r, p, q \in \mathbb{N}$ in the range
\begin{equation}
\label{range}
 0 < p, q < 2r \;, \qquad -r < p-q < r \;, \qquad r < p+q < 3r \;,
\end{equation} 
such that the following divisibility constraints hold
\begin{equation}
\label{divisibility}
 (p, \, 2r) =1 \;, \quad \quad (p+q-r, \, 2r)=1 \;, \quad \quad (p-q+r, \, 2r)=1 \;.
\end{equation} 
We will now construct a smooth irreducible curve $C$ over $\mathbb{C}(\lambda)$ associated with the curve $SE(\lambda)^{2r}_{r,p,q}$.

\begin{remark}
For $r=1$ it follows $p=q=1$ and Equation~(\ref{SuperLegendre1}) is the classical Legendre normal form of an elliptic curve. 
\end{remark}
\begin{remark}
\label{simplify}
The inequalities and divisibility constraints simplify considerably in two special cases:
in the case $q=r$, all inequalities and divisibility constraints simplify to $0<p < 2r$ with $(p,2r)=1$.
Similarly, in the case $q=3r-2p$, all inequalities and divisibility constraints simplify to $2r< 3p < 4r$ with $(3p,2r)=1$.
\end{remark}
\begin{lemma}
\label{projective_model}
The minimal resolution of the curve $SE(\lambda)^{2r}_{r,p,q}$ is a smooth irreducible curve of genus $2r-1$ 
if the integers $(r,p,q)$ satisfy the inequalities~(\ref{range}), the divisibility constraints~(\ref{divisibility}),
and $\lambda \in \mathbb{C} \backslash \lbrace 0, 1, \infty \rbrace$.
\end{lemma}
\begin{proof}
The reduced discriminant of Equation~(\ref{SuperLegendre1}) vanishes for $\lambda \in \lbrace 0, 1, \infty \rbrace$.
We consider the projective curve in $\mathbb{P}^2$ given by
\begin{equation}
\label{SuperLegendreProjective}
SE(\lambda)^{2r}_{r,p,q}: \quad Y^{2r} Z^{r-p+q} = X^{p+q-r} \; (X-Z)^{2r-p} \; (X-\lambda \, Z)^{2r-p} \;.
\end{equation}
Given $\lambda \in \mathbb{C} \backslash \lbrace 0, 1, \infty \rbrace$, a straightforward calculation 
shows that (\ref{SuperLegendreProjective}) is smooth if $r=1$ which in turn implies $p=q=1$, i.e., the curve is the classical Legendre form 
of a smooth elliptic curve. In this case, the curve (\ref{SuperLegendreProjective}) is the desired smooth projective model.

Let us now assume that $r>1$. The curve~(\ref{SuperLegendreProjective}) is singular and has at most the four isolated singular points 
\begin{equation}
\label{SingPts}
 [X:Y:Z] \in \Big\lbrace [0:0:1], [1:0:1], [\lambda:0:1], [0:1:0] \Big\rbrace \;,
\end{equation}
and is smooth everywhere else. The homogeneous polynomial 
\begin{equation}
\label{def_poly}
 F(X,Y,Z)= -Y^{2r} Z^{r-p+q} + X^{p+q-r} \; (X-Z)^{2r-p} \; (X-\lambda \, Z)^{2r-p}
\end{equation}
is irreducible. We first observe that the only point at infinity is $[X:Y:Z]= [0:1:0]$, and it is a singular point if $r-p+q >1$. A generic line through the point
$[0:1:0]$ has an equation of the from $X/\xi=Z/\kappa$ and $Y=1$. Thus, any point $[\rho\xi:1:\rho\kappa]$ on this line also lies on the curve
if $\rho$ is a root of the equation
$$
 0 = - \rho^{r-p+q} \, \kappa^{r-p+q} + \rho^{3r-p+q} \, \xi^{p+q-r} \, (\xi-\kappa)^{2r-p} \, (\xi-\lambda \, \kappa)^{2r-p} \;.
$$
This shows that $[X:Y:Z]= [0:1:0]$ is a point of multiplicity $r-p+q$. For the equation to have additional roots,
$\kappa$ must vanish and the line $[\rho\xi:1:0]$ then has $3r-p+q$ of its intersections coincident with the curve in $[0:1:0]$, i.e., it is a nodal
tangent. Moreover, because of the coefficient being $\kappa^{r-p+q}$, all of the nodal tangents coincide. It is well-known that
any irreducible plane curve can be transformed, by a finite succession of standard quadratic transformations, into a curve whose
multiple points are all ordinary, i.e., all nodal tangents are distinct at these points.  During the blow-up the singularity is then resolved
into a set of simple points on the transform corresponding to directions of the nodal tangents.  It can be shown that because of the relation
$$
 (3r-p+q,r-p+q) = (2r, r-p+q) = (2r, -r -p +q) = (2r, p-q+r) =1 
 $$ 
there is only one point above infinity on the blow-up.  In other words, the point over infinity is totally ramified. We then proceed in the same way for the other points in~(\ref{SingPts}).
The points in~(\ref{SingPts}) are found to be multiple points with multiplicities
\begin{equation}
 p+q-r, \; 2r-p, \; 2r-p, r-p+q \;.
\end{equation}
Requiring total ramification over each point leads precisely to the divisibility constraints~(\ref{divisibility}).
We thus obtain a smooth irreducible model.  

If we consider the map $\phi: SE(\lambda)^{2r}_{r,p,q} \to \mathbb{P}^1$ given by $\phi: [X:Y:Z] \mapsto [X:Z]$, then this is a map of degree $2r$ which has four ramification points
with total ramification. The Riemann-Hurwitz formula implies
\begin{equation}
 2(1-g) = 2 \, (2r)\, (1-0) - 4 \,  (2r-1)  \;.
\end{equation}
Hence, we find $g=2r-1$.
\end{proof}
To understand the resolution process from Lemma~\ref{projective_model} in more detail,
we can compute the Puiseux expansions for the defining polynomial $f$ in Equation~(\ref{def_poly}) in terms of a local parameterw $z_0, z_1, z_\lambda, z_\infty \in \mathbb{C}$ 
around each singular point given in Equation~(\ref{SingPts}). The leading terms that determine the analytic equivalence class of each singularity are listed
in Table~\ref{tab:Puiseux} where $y_0(z^{2r})$, $y_1(z^{2r})$, $y_\lambda(z^{2r})$, $x_\infty(z^{2r})$ are -- by construction -- convergent power series
in $z^{2r}$.
 \begin{table}[H]
\scalebox{0.55}{
\begin{tabular}{|c|c|c|c|c|}
\hline
&&&&\\[-0.8em]
$[X:Y:Z]$ & $[0:0:1]$ &  $[1:0:1]$ & $[\lambda:0:1]$ & $[0:1:0]$\\[0.2em]
\hline
&&&&\\[-0.8em]
multiplicity & $p+q-r$ & $2r-p$ & $2r-p$ & $r-p+q$\\[0.2em]
\hline
&&&&\\[-0.8em]
$\begin{array}{c} \text{Puiseux} \\ \text{coefficients} \end{array}$
 & $(2r,p+q-r)$ 
 & $(2r,2r-p)$ 
 & $(2r,2r-p)$ 
 & $(3r-p+q,r-p+q)$ \\[0.2em]
\hline
&&&&\\[-0.8em]
$\begin{array}{c} \text{Puiseux} \\ \text{expansion} \end{array}$
& $\begin{array}{l} X=z_0^{2r} \\[0.2em] Y = \lambda^{\frac{2r-p}{2r}} \, z_0^{p+q-r} \, \big(1 + y_0(z_0^{2r}) \big)\\[0.2em] Z=1 \end{array}$ 
& $\begin{array}{l} X= 1 + z_1^{2r} \\[0.2em] Y = (1-\lambda)^{\frac{2r-p}{2r}} \, z_1^{2r-p}  \, \big(1 + y_1(z_1^{2r}) \big)\\[0.2em] Z=1 \end{array}$ 
& $\begin{array}{l} X= \lambda + z_\lambda^{2r} \\[0em] Y = \big(\lambda (\lambda-1)\big)^{\frac{2r-p}{2r}} \, z_\lambda^{2r-p}  \, \big(1 + y_\lambda(z_\lambda^{2r}) \big) \\[0.3em] Z=1\end{array}$
& $\begin{array}{l} X= z_\infty^{r-p+q}  \, \big(1 + x_\infty(z_\infty^{2r}) \big) \\[0.2em] Y = 1 \\[0.1em] Z= z_\infty^{3r-p+q} \end{array}$  \\[0.2em]
\hline
\end{tabular}}
\caption{Puiseux expansions about singular points}\label{tab:Puiseux}
\end{table}
Each singularity is thus locally analytically equivalent to a neighborhood of the origin for the curve $x^a - y^b =0$ 
in $\mathbb{C}^2$ where $a, b$ are relatively prime integers.
By \cite[Prop.~1, p.~224]{MR646612} such a singularity is topologically equivalent to a cone over a torus knot of type $(a,b)$.
It is well-known that any singularity of an irreducible plane algebraic curve can be completely resolved by a succession of standard quadratic transformations.
The first of such blow-ups are described in \cite[Example 2, p.~471]{MR646612} in detail.
The minimal resolution and corresponding resolution graph of the curve $SE(\lambda)^{2r}_{r,p,q}$ is obtained
using a Theorem of Enriques and Chisini \cite[Thm.~12, p.~516]{MR646612} by running an Euclidean division algorithm
on the two integers $a, b$.

Next we prove that the involution automorphism on the curve $SE(\lambda)^{2r}_{r,p,q}$ given by
\begin{equation}
\label{involution_EC}
 \imath: (x,y) \mapsto (x,-y) \;,
\end{equation}
lifts to its minimal resolution constructed in Lemma~\ref{projective_model}.
\begin{lemma}
\label{Lem:lift}
The involution automorphism~(\ref{involution_EC}) lifts to the minimal resolution of the curve $SE(\lambda)^{2r}_{r,p,q}$ constructed in Lemma~\ref{projective_model}.
\end{lemma}
\begin{proof}
We have to prove that we can lift the involution in Equation~(\ref{involution_EC}) to local coordinates on the normalization of neighborhoods of the singular points.
We start by constructing the local normalization of a neighborhood of the singular point $[X:Y:Z]=[0:0:1]$ or $(x,y)=(0,0)$.
For $\epsilon > 0$ and $\delta >0$ we consider the neighborhoods
$$
 U_{\epsilon, \delta} := \left\lbrace (x,y) \in \mathbb{C}^2 \Big| \;  |x| < \delta, \, |y| < \epsilon, \, f(x,y,1)=0 \right\rbrace \;,
$$
where the defining polynomial $f$ was given in Equation~(\ref{def_poly}) and
$$
 B_{\delta} := \left\lbrace z_0\in \mathbb{C} \Big| \;  |z_0| < \delta^{\frac{1}{2r}}  \right\rbrace \;.
$$
By the local normalization theorem \cite[Thm.~1]{MR646612},  there is and $\epsilon_0>0$ such that for each $0 < \epsilon < \epsilon_0$ there is a $\delta>0$ 
such that the mapping $\pi_0: B_{\delta} \to \mathbb{C}^2$ with 
$$\pi_0(z_0)=\big(z_0^{2r}, \lambda^{2r-p} \, z_0^{p+q-r} \, \big(1 + y_0(z_0^{2r})\big)$$
is holomorphic and onto $U_{\epsilon, \delta}$. Moreover, the restriction 
$$ 
 \pi_0: \; B_{\delta} \backslash \lbrace 0 \rbrace \to U_{\epsilon, \delta}   \backslash \lbrace 0 \rbrace 
$$
is biholomorphic with $\pi_0^{-1}(0)=0$. Note that the action $z_0 \mapsto -z_0$ on $B_{\delta}$ after projection
matches the involution in Equation~(\ref{involution_EC}) since $p+q-r$ is odd which follows from the condition $(p+q-r,2r)=1$.
Similarly, we proceed for  the singular points at $(x,y)=(1,0)$ and $(x,y)=(\lambda,0)$. 

For the singular point $[X:Y:Z]=[0:1:0]$, we note that the involution on the curve $SE(\lambda)^{2r}_{r,p,q}$ in projective coordinates is given by
\begin{equation}
\label{involution_EC2}
 \imath: [X:Y:Z] \mapsto [-X:Y:-Z] \;.
\end{equation}
The local normalization theorem then provides a local coordinate $z_\infty$ on the normalization and the biholomorphic mapping 
$$\pi_\infty(z_\infty)=\big(z_\infty^{r-p+q} \, (1+x_\infty(z^{2r}), \, z_\infty^{3r-p+q} \big) = \left(\frac{X}{Y}, \frac{Z}{Y}\right) \;.$$
The action $z_\infty \mapsto -z_\infty$  after projection
matches the involution automorphism in Equation~(\ref{involution_EC2}) since $r-p+q$ and $3r-p+q$ are odd.
\end{proof}
Note that Lemma~\ref{Lem:lift} does not guarantee that the action of the bigger group $\mathbb{Z}_{2r}: (x,y) \mapsto (x, \, \rho_{2r} y)$ lifts 
to the minimal resolution for general integers $p, q, r$ with $\rho_{2r}=\exp{(\frac{2\pi i}{2r})}$ as well. However, in the special case $q=r$ we can
give lift the action and give a precise description on the resolution curve.
\begin{lemma}
\label{Lem:lift2}
The action $\mathbb{Z}_{2r}: [X:Y:Z]\mapsto [X: \, \rho^p_{2r} Y:Z]$ lifts to the minimal resolution of the curve $SE(\lambda)^{2r}_{r,p,r}$ constructed in Lemma~\ref{projective_model}.
\end{lemma}
\begin{proof}
Using $q=r$ and $(p,2r)=1$, the local coordinates $z_0$, $z_1$, $z_\lambda$, and $z_\infty$ from the proof of Lemma~\ref{Lem:lift},
and the Puiseux expansions in Table~\ref{tab:Puiseux}, we obtain the following actions on the local normalizations:
$$
 \begin{array}{ccl} z_0 & \mapsto & \rho_{2r} z_0 \\[0.4em]
 X & \mapsto & X \\[0.2em]
 Y & \mapsto & \rho^{p}_{2r} Y\\[0.4em]
 Z & \mapsto & Z
 \end{array}
  , \; 
   \begin{array}{ccl} z_1, z_\lambda & \mapsto &  \rho^{2r-1}_{2r} z_1, z_\lambda \\[0.4em]
 X & \mapsto & X \\
 Y & \mapsto & \left(\rho^{2r-1}_{2r}\right)^{2r-p} Y = \rho^{p}_{2r} Y\\[0.4em]
 Z & \mapsto & Z
 \end{array}
  , \;
  \begin{array}{ccl}  z_\infty  & \mapsto & \rho_{2r} z_\infty  \\[0.2em]
 X & \mapsto &  \left(\rho_{2r}\right)^{2r-p} X=  \rho_{2r}^{-p} X\\[0.4em]
 Y & \mapsto & Y\\
 Z & \mapsto & \left(\rho_{2r}\right)^{4r-p} Z=  \rho_{2r}^{-p} Z
 \end{array} 
 $$ 
 This proves that the actions on the local normalizations patch together to give a lift of the action
 $$
  [X:Y:Z]\mapsto [X: \, \rho^p_{2r} Y:Z] = [\, \rho^{-p}_{2r} X:  Y: \, \rho^{-p}_{2r} Z] \;.
$$  
\end{proof}

\begin{definition}
\label{WhatIsSuperelliptic}
We call the smooth irreducible curve of genus $2r-1$ that is the minimal resolution of the algebraic curve $SE(\lambda)^{2r}_{r,p,q}$ a superelliptic curve
assuming that the inequalities~(\ref{range}), the divisibility constraints~(\ref{divisibility}), and  $\lambda \in \mathbb{C} \backslash \lbrace 0, 1, \infty \rbrace$ are satisfied.
\end{definition}

\begin{remark}
\label{Rem:AffineModel2}
All statements derived so far will also be valid for the smooth irreducible over $\mathbb{C}(\lambda)$ 
associated with the affine curve of degree $3r+p-q$ given by
\begin{equation}
\label{SuperLegendre2}
SE(\lambda)^{2r}_{r,2r-p,2r-q}: \quad y^{2r} = x^{3r-p-q} \; (x-1)^{p} \; (x-\lambda)^{p} \;.
\end{equation}
The affine curve~(\ref{SuperLegendre2}) is obtained from Equation~(\ref{SuperLegendre1}) by sending $(p,q)\mapsto (2r-p,2r-q)$ which leaves the inequalities~(\ref{range})
and divisibility constraints~(\ref{divisibility}) unchanged.
\end{remark}

\begin{lemma}
\label{SWmodel}
The affine curve $SE(\Lambda^2)^{2r}_{r,2r-p,2r-q}$ in Equation~(\ref{SuperLegendre2}) is equivalent to the affine curve
\begin{equation}
\label{Legendre}
  y^{2r} = x^{3r-p-q} \; \left(x^2 + 2 \, \big( 1 - 2\, u\big) \, x +1 \right)^{p} 
\end{equation}
with  $u=\frac{(1+\Lambda)^2}{4 \, \Lambda}$.
Moreover, the affine curves $SE(\lambda_1)^{2r}_{r,q,r}$ and $SE(\Lambda^2_2)^{2r}_{r,r,2r-q}$ are related by a rational transformation with
\begin{equation}
\label{2IsoParam}
  \lambda_1 = \left( \frac{1+\Lambda_2}{1-\Lambda_2} \right)^2 \;.
\end{equation}  
\end{lemma}
\begin{fact} \label{ModularSurfaces}
For $r=p=q=1$, the pencil of curves in Equation~(\ref{Legendre}) is the modular elliptic surface for $\Gamma_0(4)$,
and the pencil of curves in Equation~(\ref{SuperLegendre2}) is the modular elliptic surface for $\Gamma(2)$.
Lemma~\ref{SWmodel} then describes the (invertible) action of the translation by an order-two point, i.e., two-isogeny.
\end{fact}
\begin{proof}
The affine curve
\begin{equation}
SE(\Lambda^2_2)^{2r}_{r,2r-p,2r-q}: \quad \eta_2^{2r} = \zeta_2^{3r-p-q} \; (\zeta_2-1)^{p} \; (\zeta_2-\Lambda_2^2)^{p} 
\end{equation}
is related to the affine curve in Equation~(\ref{Legendre}) by the birational transformation
\begin{equation}
\label{TransfoLegendre1}
 x = \frac{\zeta_2}{\Lambda_2} \;, \quad y = \frac{\eta_2}{\Lambda_2^{\frac{3}{2}+\frac{p}{2r}-\frac{q}{2r}}} \;, \quad u = \frac{(1+\Lambda_2)^2}{4 \, \Lambda_2}  \;.
\end{equation}
In turn, the affine curve~(\ref{Legendre}) is related to the affine curve
\begin{equation}
\label{test}
  \eta_1^{2r} = \zeta_1^{r-p+q} \; (\zeta_1-1)^{3r-p-q} \; (\zeta_1-\lambda_1)^{3r-p-q} \;   (\zeta_1^2 -\lambda_1)^{2(p-r)} 
\end{equation}
by the rational  transformation
\begin{equation}
\label{TransfoLegendre2}
 x = \frac{(\zeta_1-1) \, (\zeta_1-\lambda_1)}{\zeta_1 \, (1-\lambda_1)} \;, \quad y = \frac{(\zeta_1^2-\lambda_1) \, \eta_1}{\zeta_1^2 \, (1-\lambda_1)^{\frac{3}{2}+\frac{p}{2r}-\frac{q}{2r}}} \;, \quad 
 u = \frac{\lambda_1}{\lambda_1-1}  \;.
\end{equation}
For $p=r$, the affine curve~(\ref{test}) coincides with $SE(\lambda_1)^{2r}_{r,q,r}$. Therefore, a rational transformation between the affine curves $SE(\lambda_1)^{2r}_{r,q,r}$ and $SE(\Lambda^2_2)^{2r}_{r,r,2r-q}$
is given by
\begin{equation}
\label{TransfoLegendre3}
 \zeta_2 = - \frac{(\zeta_1-1) \, (\zeta_1-\lambda_1) \, (1-\Lambda_2)^2}{4 \, \zeta_1} \;, \quad
  \eta_2 =  \frac{(-1)^{{\frac{3}{2}+\frac{p}{2r}-\frac{q}{2r}}} \,(\zeta_1^2-\lambda_1) \, (1-\Lambda_2)^{3+\frac{p}{r}-\frac{q}{r}} \, \eta_1}
 {2^{3+\frac{p}{r}-\frac{q}{r}}\, \zeta_1^2 \,  }
\end{equation}
with $\lambda_1 = ( \frac{1+\Lambda_2}{1-\Lambda_2})^2$.
\end{proof}
\begin{remark}
Given the transformation~(\ref{2IsoParam}) between the parameters, the relation between the affine curves $SE(\lambda_1)^{2r}_{r,q,r}$ and $SE(\Lambda^2_2)^{2r}_{r,r,2r-q}$ is governed by
the following $(8r-q,4r-q)$-correspondence in $(\zeta_1, \zeta_2)$:
\begin{equation}
\begin{split}
   (-1)^q \, 4^{-4\,r+q} \, \zeta_1^{q} \,  \left( \zeta_1-1 \right) ^{2\,r-q} \, \left( \zeta_1 -\lambda_1 \right) ^{2\,r-q} \, ( \zeta_1^{2}-\lambda_1 ) ^{2\,r}\\
  = \; \, \zeta_1^{4r} \, \zeta_2^{2\,r-q} \, \left( \zeta_2-1 \right) ^{r} \, ( \zeta_2 - \Lambda_2^2 ) ^{r} \, \left( \Lambda_2 -1 \right)^{-8\,r+2\,q} \;.
\end{split}
\end{equation}
\end{remark}

\subsection{Periods for superelliptic curves}
\label{SECperiods}
As pointed out in the proof of Lemma~\ref{projective_model}, for $y$ to be a well-defined multi-valued function on $SE(\lambda_1)^{2r}_{r,p,q}$, we had to introduce branch cuts which we chose
to be the line segments $[e_1,e_2]$ and $[e_3,e_4]$ where $e_1, \dots, e_4$ are the four branch points in~(\ref{SingPts}). In the previous section, we also explained
that the minimal resolution $C_1$ of the curve $SE(\lambda)^{2r}_{r,p,q}$ replaces the singular points by trees of rational curves which does not
change the rank of the first homology by van Kampen's theorem. Therefore, a basis of one-cycles for the curve $C_1$ of genus $2r-1$ in Lemma~\ref{projective_model} can 
be constructed by using the curve $SE(\lambda_1)^{2r}_{r,p,q}$ directly.

First we consider the case $r=1$, i.e., the case of an elliptic curve. In this case a 
basis of one-cycles is given by the so-called $A$- and $B$-cycle. The $A$-cycle $\mathfrak{a}_1$ is a closed clockwise cycle around 
the line segment $[e_1,e_2]$ (on one of the two $y$-sheets) not cutting through the second branch cut. The $B$-cycle $\mathfrak{b}_1$  is the closed cycle that 
runs from the first to the second branch cut on the chosen $y$-sheet and returns on the adjacent sheet -- in this case the $(-y)$-sheet -- in an orientation that gives the positive intersection number
$\mathfrak{a}_1\circ \mathfrak{b}_1=1$. The situation is depicted in Figure~\ref{ABcycles}.
\begin{figure}[ht!]
\includegraphics[scale=0.9]{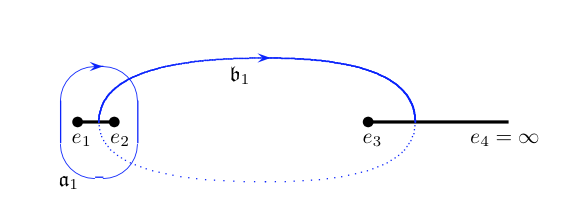}
\caption{$A$- and $B$-cycle on double-branched cover}
\label{ABcycles}
\end{figure}

For the superelliptic curve $SE(\lambda_1)^{2r}_{r,p,q}$ there are $2r$ choices for the $y$-sheet in the above construction.
Therefore, the construction leads to $2r-1$ $A$-cycles $\mathfrak{a}_i$ and $B$-cycles $\mathfrak{b}_j$ for $1\le i, j \le 2r-1$ with intersection numbers
$$
 \mathfrak{a}_i\circ \mathfrak{a}_{j}= 0 \;, \quad  \mathfrak{b}_i\circ \mathfrak{b}_{j}= 0 \;,  \quad \mathfrak{a}_i\circ \mathfrak{b}_{j}= \delta_{i,j} - \delta_{i,j+1} \;.
$$ 
Similar to the elliptic-curve case, there are $2r-1$ and not $2r$ cycles, as the last choice gives a cycle that is already a  linear combination of the cycles already constructed.

We now want to describe the integrals of the first kind explicitly. Recall that in Lemma~\ref{projective_model} we constructed the resolution $\psi_1: C_1 \to SE(\lambda_1)^{2r}_{r,p,q}$
of the plane algebraic curve $SE(\lambda_1)^{2r}_{r,p,q}$ with affine equation $F(x,y)=0$ where
\begin{equation}
\label{def_poly_affine}
 F(x,y)= -y^{2r} + x^{p+q-r} \; (x-1)^{2r-p} \; (x-\lambda )^{2r-p} \;.
\end{equation}
The affine coordinates $x, y$ are rational functions on $SE(\lambda_1)^{2r}_{r,p,q}$, and hence give rational functions on $C_1$ when composed with $\psi_1$.
Likewise, their differentials $dx$ and $dy$ are meromorphic differentials on $C_1$. We have the following lemma:
\begin{lemma}
\label{Lem:1-form}
The differential form $dx/y$ is a holomorphic differential one-form on $C_1$.
\end{lemma}
\begin{proof}
We will first discuss the situation for all points in Equation~(\ref{SingPts}) with $Z\not= 0$. The form $dx/y$ has poles, if at all, at the points of $SE(\lambda_1)^{2r}_{r,p,q}$ with $y=0$.
These are precisely the singular points. Let $t$ be the local coordinate on $C_1$ in a neighborhood of the point which is mapped
onto any of the singular points by the resolution $\psi_1: C_1 \to SE(\lambda_1)^{2r}_{r,p,q}$. In terms of the local coordinate $t$ the resolution is described by the two functions
$x=x(t)$ and $y=y(t)$ given in Table~\ref{tab:Puiseux}, and the differential form $dx/y$ equals
\begin{equation}
 \frac{dx}{y} = \frac{\dot{x}(t) \, dt}{y(t)} \;.
\end{equation}
Hence, the one-form has no pole iff the orders of these functions satisfy
\begin{equation}
\label{ineq}
 \big(O(x \circ \psi_1) - 1\big)  - O(y \circ \psi_1) \ge 0 \;.
\end{equation}
For the point $(x,y)=(0,0)$, Equation~(\ref{ineq}) is equivalent to $3r>p+q$. For the point $(x,y)=(1,0)$ and $(x,y)=(\lambda,0)$, Equation~(\ref{ineq}) is equivalent to $p>0$.
The inequalities are satisfied by assumption~(\ref{range}). To discuss the point at infinity, we substitute $x=X/Z$ and $y=Y/Z$.  In terms of the local coordinate $t$ the resolution 
is described $X=X(t)$, $Y=1$, and $Z=Z(t)$, and we obtain the differential form
$$
 \frac{dX}{Y} - \frac{X \, dZ}{Y \, Z} = \big( \dot{X}(t) - \frac{X(t) \, \dot{Z}(t)}{Z(t)} \big) \, dt\;.
$$ 
Hence, the one-form has no pole iff 
\begin{equation}
\label{ineq2}
 \big(O(X \circ \psi_1) - 1\big)  \ge 0 
\end{equation}
or $r>p-q$ which is satisfied by assumption~(\ref{range}).
\end{proof}
\begin{remark}
Notice that at each smooth point of the algebraic curve $SE(\lambda_1)^{2r}_{r,p,q}$ with affine equation $F(x,y)=0$ of degree $d=3r-p+q$ we
can compute its total differential. We obtain
\begin{equation}
\label{one-form}
 F_x \, dx - 2r \, y^{2r-1} \, dy = 0  \; \Rightarrow \; \frac{dx}{y} = 2r \, y^{2r-2} \, \frac{dy}{F_x} \;.
\end{equation}
Because of the inequalities~(\ref{range}) we have $0 \le 2r-2 \le d-3$. Equation~(\ref{one-form})
then gives an equivalent form for the holomorphic one-form $dx/y$ that already appeared in the work of Riemann \cite{MR1579035}.
\end{remark}

We have proved that for the purpose of a period computation there will be no danger from using the curve $SE(\lambda)^{2r}_{r,p,q}$ and treating the points in~(\ref{SingPts}) as formal symbols.
We then have the following lemma computing the periods of the holomorphic one-form $dx/y$ for the resolution curves $C_1$ and $C_2$ of $SE(\lambda)^{2r}_{r,p,q}$ and $SE(\lambda)^{2r}_{r,2r-p,2r-q}$, respectively:
\begin{lemma}
\label{SuperellipticPeriods}
Setting $\beta_1=q/(2r)$ and $\beta_2=p/(2r)$, we obtain for the periods of $dx/y$ on $C_1$ 
\begin{equation}
\label{period2}
\begin{split}
f^{(k)}_{r,p,q}(\lambda)= \oint_{\mathfrak{a}_k} \frac{dx}{y}   & =  C^{(k)}_{2r} \, \frac{ \Gamma\big( \frac{3}{2}-\beta_1-\beta_2\big) \, \Gamma(\beta_2) \,\lambda^{\frac{1}{2}-\beta_1}}{\Gamma\big(\frac{3}{2}-\beta_1\big)} \;   
   \hpg21{\frac{3}{2}-\beta_1-\beta_2,\,1-\beta_2}{\frac{3}{2}-\beta_1}{\lambda}   \;,\\
f^{(k) \; \prime}_{r,p,q}(\lambda)=  \oint_{\mathfrak{b}_k} \frac{dx}{y}   & = (-1)^{\beta_2} \, C^{(k)}_{2r} \, \frac{ \Gamma(\beta_2)^2}{\Gamma(2\beta_2) \; (1-\lambda)^{1-2\beta_2}} \;   
   \hpg21{\beta_1+\beta_2-\frac{1}{2},\,\beta_2}{2\beta_2}{1-\lambda}   \;,
 \end{split}  
\end{equation}
and on $C_2$ 
\begin{equation}
\label{period1}
\begin{split}
 f^{(k)}_{r,2r-p,2r-q}(\lambda)
 & = C^{(k)}_{2r} \, \frac{\Gamma\big(\beta_1 + \beta_2 -\frac{1}{2}\big) \, \Gamma(1-\beta_2)}{\Gamma\big(\beta_1 + \frac{1}{2}\big) \; \lambda^{\frac{1}{2}-\beta_1}} \; 
  \hpg21{\beta_1+\beta_2-\frac{1}{2},\,\beta_2}{\beta_1 + \frac{1}{2}}{\lambda}   \;, \\
f^{(k) \; \prime}_{r,2r-p,2r-q}(\lambda)
& =  (-1)^{1-\beta_2} \, C^{(k)}_{2r} \,  \frac{ \Gamma\big(1 - \beta_2 \big)^2  \, (1-\lambda)^{1-2\beta_2}}{\Gamma\big(2 - 2\beta_2 \big) } \; 
  \hpg21{\frac{3}{2} - \beta_1 -\beta_2,\,1-\beta_2}{2-2\beta_2}{1-\lambda}   \;
\end{split}  
\end{equation}
with $C^{(k)}_{2r}= (\rho_{2r}-1)/\rho_{2r}^{k}$, $(-1)^{\beta_2}= \rho_{4r}^{p}$, $\rho_{2r}=\exp{(\frac{2\pi i}{2r})}$ for $k=1, \dots, 2r-1$.
\end{lemma}
\begin{remark}
Whenever $k=1$, i.e., we are looking at the principal branch, we will drop the superscript and write $f_{r,p,q}(\lambda), f'_{r,p,q}(\lambda)$, etc.
\end{remark}
\begin{remark}
For $p=q=r=k=1$, the above lemma recovers the well-known period relations for an elliptic curve
\begin{equation}
 \oint_{\mathfrak{a}_1} \frac{dx}{y}    =  2 \pi \,  \hpg21{\frac{1}{2}, \frac{1}{2}}{1}{\lambda}   \;,\quad
 \oint_{\mathfrak{b}_1} \frac{dx}{y}    = 2 \pi i \,  \hpg21{\frac{1}{2}, \frac{1}{2}}{1}{1-\lambda}  \;.
\end{equation}
\end{remark}
\begin{proof}
First, we note that Equations~(\ref{period1}) can be obtained from Equations~(\ref{period2}) by sending $(\beta_1,\beta_2) \mapsto (1-\beta_1,1-\beta_2)$.
Therefore, we only have to prove Equations~(\ref{period2}).
We choose $\lambda \in \mathbb{R}$ with $0 <  \lambda< 1$, and $e_1=0$, $e_2=\lambda$, $e_3=1$, and $e_4=\infty$. 
For $x=\delta_1+i \, \delta_2$ with $0<\delta_1<1$, we have for $y$ in $SE(\lambda)^{2r}_{r,p,q}$
$$
  \lim_{\delta_2 \to 0 \pm} \operatorname{sig}\Big(\re(y^{2r})\Big) = 1 \;,  \qquad    \lim_{\delta_2 \to 0 \pm} \im(y^{2r}) = 0\;,  \qquad    \lim_{\delta_2 \to 0 \pm} \operatorname{sig}\Big(\im(y^{2r})\Big) = \pm 1 \;.
$$
We then define a multi-valued function $y$ on a $2r$-fold branched domain by writing
$$
 y = x^{\beta_1+\beta_2-\frac{1}{2}} \, (1-x)^{1-\beta_2} \, (\lambda-x)^{1-\beta_2} \;.
$$
We reduce the integration over the $A$-cycle to an integration along the branch cut and back after inserting the correct phase. We obtain
\begin{equation}
\begin{split}
\oint_{\mathfrak{a}_1} \frac{dx}{y}   & = \left(1 - e^{-\frac{2\pi i}{2r}}\right) \, \int_0^\lambda \frac{dx}{x^{\beta_1+\beta_2-\frac{1}{2}} \, (1-x)^{1-\beta_2} \, (\lambda-x)^{1-\beta_2}} \\
& = \frac{ \left(1 - e^{-\frac{2\pi i}{2r}}\right)}{\lambda^{\beta_1-\frac{1}{2}}}   \;  \int_0^1 \frac{d\tilde{x}}{\tilde{x}^{\beta_1+\beta_2-\frac{1}{2}} \, (1-\lambda \, \tilde{x})^{1-\beta_2} \, (1-\tilde{x})^{1-\beta_2}}\\
& = \frac{ \left(1 - e^{-\frac{2\pi i}{2r}}\right)}{\lambda^{\beta_1-\frac{1}{2}}} \;  \frac{ \Gamma\big( \frac{3}{2}-\beta_1-\beta_2\big) \, \Gamma(\beta_2)}{\Gamma\big(\frac{3}{2}-\beta_1\big)} \;   
   \hpg21{\frac{3}{2}-\beta_1-\beta_2,\,1-\beta_2}{\frac{3}{2}-\beta_1}{\lambda} \;.
\end{split}
\end{equation}
If $\mathfrak{a}_1$ is replaced by $\mathfrak{a}_{k}$, $y$ must be replaced by $\rho_{2r}^{k-1} y$ with $\rho_{2r}=\exp{\frac{2\pi i}{2r}}$ to account
for the change in $y$-sheet. 

For $x=\delta_1+i \, \delta_2$ with $1<\delta_1<\lambda$ and $\delta_2 \approx 0$, we have for $y$ in $SE(\lambda)^{2r}_{r,p,q}$ that
$(-1)^{2r-p}\re{(y^{2r})}>0$. We recall that the $B$-cycle $\mathfrak{b}_1$ was chosen as the closed cycle that 
runs from the first to the second branch cut on the chosen $y$-sheet and returns on the adjacent $y$-sheet as shown in Figure~\ref{ABcycles}.
This implies that
\begin{equation}
\begin{split}
\oint_{\mathfrak{b}_1} \frac{dx}{y}   & =\left(1- \frac{1}{\rho_{2r}}\right) \,\int_\lambda^1 \frac{dx}{x^{\beta_1+\beta_2-\frac{1}{2}} \, (1-x)^{1-\beta_2} \, (\lambda-x)^{1-\beta_2}} \;.
\end{split}
\end{equation}
Setting $x=1-(1-\lambda)\tilde{x}$, we obtain
\begin{equation}
\begin{split}
\oint_{\mathfrak{b}_1} \frac{dx}{y}   & =  e^{-\pi i\,(1-\beta_2)}\,  \frac{ \left(1 - e^{-\frac{2\pi i}{2r}}\right)}{(1-\lambda)^{1-2\beta_2}} \, \int_1^0 \frac{d\tilde{x}}{(1-(1-\lambda)\tilde{x})^{\beta_1+\beta_2-\frac{1}{2}} \, \tilde{x}^{1-\beta_2} \, (1-\tilde{x})^{1-\beta_2}} \\
&= e^{\pi i\,\beta_2}\,    \left(1 - e^{-\frac{2\pi i}{2r}}\right)\,  \frac{ \Gamma(\beta_2)^2}{\Gamma(2\beta_2) \; (1-\lambda)^{1-2\beta_2}} \;   
   \hpg21{\beta_1+\beta_2-\frac{1}{2},\,\beta_2}{2\beta_2}{1-\lambda} \;.
\end{split}
\end{equation}
If $\mathfrak{b}_1$ is replaced by $\mathfrak{b}_{k}$, $y$ must be replaced by $\rho_{2r}^{k-1} y$ with $\rho_{2r}=\exp{\frac{2\pi i}{2r}}$ to account
for the change in $y$-sheet. 

Notice that our choice of branch cuts between $e_1=0$ and $e_2=\lambda$ and $e_3=1$ and $e_4=\infty$, respectively, is compatible with the definition 
in Section~\ref{Appell} that the complex continuation of the Gauss' hypergeometric function has a branch cut from $1$ to $\infty$ on the real $\lambda$-axis.
Therefore, all formulas remain consistent if we allow $\lambda$ to be complex.
\end{proof}
\begin{remark}
\label{Rem:integrals}
In the above proof we have used that on $SE(\lambda)^{2r}_{r,p,q}$ we have
\begin{equation}
\begin{split}
& f^{(k) \; \prime}_{r,p,q}(\lambda)=  \oint_{\mathfrak{b}_k} \frac{dx}{y}   = C^{(k)}_{2r} \,  \int_1^\lambda \frac{dx}{y} \\
 =  (-1)^{\beta_2} \, C^{(k)}_{2r}\, & \;  \frac{ \Gamma(\beta_2)^2}{\Gamma(2\beta_2) \; (1-\lambda)^{1-2\beta_2}} \;   
   \hpg21{\beta_1+\beta_2-\frac{1}{2},\,\beta_2}{2\beta_2}{1-\lambda}  \;,
 \end{split}  
\end{equation}
and on $SE(\lambda)^{2r}_{r,2r-p,2r-q}$ we have
\begin{equation}
\begin{split}
& f^{(k)}_{r,2r-p,2r-q}(\lambda)=\oint_{\mathfrak{a}_k} \frac{dx}{y}   = C^{(k)}_{2r}  \, \int_0^\lambda \frac{dx}{y} \\
 = C^{(k)}_{2r} \,& \;  \frac{\Gamma\big(\beta_1 + \beta_2 -\frac{1}{2}\big) \, \Gamma(1-\beta_2)}{\Gamma\big(\beta_1 + \frac{1}{2}\big) \; \lambda^{\frac{1}{2}-\beta_1}} \; 
  \hpg21{\beta_1+\beta_2-\frac{1}{2},\,\beta_2}{\beta_1 + \frac{1}{2}}{\lambda}   \;.
\end{split}  
\end{equation}
\end{remark}

\begin{remark}
We introduce a $\tau$-parameter as ratio of periods over the $A$- and $B$-cycle, i.e.,
\begin{equation}
 \tau  := \frac{f'_{r,p,q}(\lambda)}{f_{r,p,q}(\lambda)} = \dots = \frac{f^{(k) \; \prime}_{r,p,q}(\lambda)}{f^{(k)}_{r,p,q}(\lambda)} \;
\end{equation}
for $k=1, \dots, 2r-1$.  Note that for $r=1$ and $\lambda=\Lambda^2$ the above equation is the well-known relation that express 
$\tau$ as ratio of complete elliptic integrals in $\Lambda$, i.e.,
\begin{equation}
 2 \pi i  \, \tau := -\frac{2 \pi \, K'(\Lambda)}{K(\Lambda)} = 4 \, \ln{\Lambda} - 8 \ln{2} + \Lambda^2 + \frac{13}{32} \, \Lambda^4 + O(\Lambda^6) \;,
 \end{equation} 
and,  by inversion, $\Lambda$ as ratio of Jacobi theta-function in $q=\exp{(2\pi i \tau)}$, i.e., 
\begin{equation}
 \Lambda =  \left(\frac{\vartheta_2(\tau)}{\vartheta_3(\tau)}\right)^2  = 4 \, q^{\frac{1}{4}} \, \left( 1 - 4 \, q^{\frac{1}{2}} + 14 \, q + O(q^{\frac{3}{2}} ) \right) \;.
\end{equation} 
In physics, this is also known as construction of the mirror map \cite{MR1370673}.
\end{remark}

\section{Generalized Kummer surface}
\label{Sec:Kummer}
In this section,  we will construct a generalized Kummer variety 
as minimal nonsingular model of a product-quotient surface with only rational double points from a pair of two superelliptic curves, determine its Hodge diamond, and 
the explicit defining equations for various superelliptic fibrations on it. 

As we mentioned in the introduction, Ernst Kummer was the first to study irreducible nodal surface of degree four in $\mathbb{P}^3$ with a maximal possible number of sixteen double points \cite{MR1579281}.
Each such surface is the quotient of a principally polarized Abelian surface $T$ -- the Jacobian $\operatorname{Jac}(C)$ of a smooth hyperelliptic curve $C$ of genus two -- 
by an involution automorphism. The involution has sixteen fixed points, namely the sixteen two-torsion points of the Jacobian. 
These, it turns out, are exactly the sixteen singular points on the Kummer surface. Resolving these sixteen rational double point singularities gives 
a $K3$ surface with sixteen disjoint rational curves. 

Kummer surfaces have a rich symmetry, the so-called $16_6$-configuration \cite{MR1097176}. There are two sets of sixteen $(-2)$-curves\footnote{
A $(-2)$-curve is a smooth rational curve whose self-intersection number is -2.} which we will label $\lbrace Z_{i,j} \rbrace$
and $\lbrace \sigma_{i,j} \rbrace$ with $i, j \in \lbrace 1, \dots, 4 \rbrace$ such that $Z_{i,j}$ and $\sigma_{k,l}$ intersect if and only if $i=k$ and $j=l$ but not both.
The $16$ curves $\lbrace Z_{i,j} \rbrace$ are the exceptional  divisors corresponding to blow-up of the 16 two-torsion points of the Jacobian, while $\lbrace \sigma_{i,j} \rbrace$
arise from embedding $C$ into $\operatorname{Jac}(C)$ as symmetric theta divisors (cf.~\cite{MR0357410}). Using curves in the $16_6$-configuration, one can define various 
elliptic fibrations on the Kummer surface, since all irreducible components of a reducible fiber in an elliptic fibration are $(-2)$-curves \cite{MR0184257}.

If we specialize to the case where the Abelian surface $T=E_1 \times E_2$ is the product of two non-isogenous elliptic curves, we obtain from the minimal nonsingular model of the quotient of $T$ 
by the involution automorphism a $K3$ surface of Picard-rank $18$. Now, there is a configuration of twenty-four $(-2)$-curves, called the double Kummer pencil. It consists of the $16$ aforementioned 
exceptional curves $\lbrace E_{i,j}\rbrace$, plus $8$ curves 
obtained as the images of $\lbrace S_i \times E_2\rbrace$ or $\lbrace E_1 \times S'_j\rbrace$. Here, $\lbrace S_i\rbrace$  and $\lbrace S'_j \rbrace$ denote the two-torsion points on $E_1$ and $E_2$,
respectively. Oguiso classified all eleven inequivalent elliptic fibrations arising on such $K3$ surfaces  \cite{MR1013073}.
The simplest of such elliptic fibrations are the ones induced by the projection of $T$ onto its first or second factor and are called the first or second Kummer pencil. 

\subsection{Generalized Kummer surfaces}
\label{SSec:SuperellipticKummer}

In Lemma~\ref{projective_model}  we constructed two smooth irreducible curves $C_1$ and $C_2$ of genus $2r-1$ as minimal resolutions
of the two curves $SE(\lambda_1)^{2r} _{r,p,q}$ and $SE(\lambda_2)^{2r}_{r,2r-p,2r-q}$ where
\begin{equation}
\label{ReminderSEs}
 \begin{split}
 SE(\Lambda_1^2)^{2r}_{r,p,q}:&  \quad \eta_1^{2r} = \zeta_1^{p+q-r} \; (\zeta_1-1)^{2r-p} \; (\zeta_1-\Lambda_1^2)^{2r-p} \;,\\
 SE(\Lambda_2^2)^{2r}_{r,2r-p,2r-q}:& \quad \eta_2^{2r} = \zeta_2^{3r-p-q} \; (\zeta_2-1)^{p} \; (\zeta_2-\Lambda_2^2)^{p} \;.
 \end{split}
\end{equation}
We now consider the action of the group $G_0=\mathbb{Z}_{2}$ on the product $SE(\lambda_1)^{2r}_{r,p,q} \times SE(\lambda_2)^{2r}_{r,2r-p,2r-q}$ given by the involution automorphism
\begin{equation}
\label{involution}
 \imath: \;  (\zeta_1,\eta_1,\zeta_2,\eta_2) \mapsto (\zeta_1,-\eta_1,\zeta_2,-\eta_2) \;.
 \end{equation}
Using Lemma~\ref{Lem:lift} on both factors, it follows that the action~(\ref{involution}) lifts to the product surface $T= C_1 \times C_2$. 
In the special case $q=r$, we will also consider the action of the group $G=\mathbb{Z}_{2r}$ on  $SE(\lambda_1)^{2r}_{r,p,r} \times SE(\lambda_2)^{2r}_{r,2r-p,r}$ generated by
\begin{equation}
\label{involution2}
 (\zeta_1,\eta_1,\zeta_2,\eta_2) \mapsto (\zeta_1,\, \rho_{2r}^p \eta_1,\zeta_2, \rho_{2r}^{-p} \eta_2) \;.
 \end{equation}
Using Lemma~\ref{Lem:lift2} on both factors, it follows that the action~(\ref{involution2}) lifts to the product surface $T= C_1 \times C_2$ as well. 
The action of $G_0$ (and $G$ in the special case $q=r$) on $T$ is diagonal, but not free. The group $G_0$ (and $G$ in the special case $q=r$) 
has  $16$ fixed points that are obtained by combining the pre-images of the singular points on $SE(\lambda_1)^{2r}_{r,p,q}$ and $SE(\lambda_2)^{2r}_{r,2r-p,2r-q}$
in Equation~(\ref{SingPts}).  We will now consider the product-quotient surface $V_0 = T/G_0$  (and  $V = T/G$  in the special case $q=r$).
The singularities of $V_0$ and $V$ are by construction finite cyclic quotient singularities, which are rational singularities.
We have the following lemma:
\begin{lemma}
\label{Lem:RDPsingularities}
The $16$~cyclic quotient singularities on $V_0$ are ordinary double points. 
Among the $16$~cyclic quotient singularities on $V$ are $8$~rational double points of type $A_{2r-1}$
and $8$~cyclic quotient singularities of type $\frac{1}{2r} (1 ,1)$.
\end{lemma}

\begin{proof} 
Consider a singular point on either $V_0$ or $V$. Then, a neighborhood of the singular point is equivalent to the quotient of $\mathbb{C}^2$ by the action of
a linear diagonal automorphism with eigenvalues $\lbrace\exp{(\frac{2\pi i a}{n})}, \exp{(\frac{2\pi i b}{n})}\rbrace$ with $(a,n)=(b,n)=1$
and $n=2r$ or $n=2$ if $p$ is on $V$ or $V_0$, respectively. In the case of $V_0$, the proof of Lemma~\ref{Lem:lift} shows
that the action of $G_0$ on the local coordinates of the normalization of $SE(\lambda_1)^{2r}_{r,p,q}$  is given by
\begin{equation}
\label{LocalActionG0}
 z_0, z_1, z_\lambda, z_\infty \mapsto -z_0, -z_1, -z_\lambda, -z_\infty \;.
\end{equation}
As pointed out in Remark~\ref{Rem:AffineModel2}, the second affine model $SE(\lambda_2)^{2r}_{r,2r-p,2r-q}$ is obtained by sending $(p,q)\mapsto (2r-p,2r-q)$.
It is easy to check that action of $G_0$ on the local coordinates of the normalization of $SE(\lambda_2)^{2r}_{r,2r-p,2r-q}$ is the same as the one in Equation~(\ref{LocalActionG0}).
Therefore, we have only rational double points of type $A_1$ on $V_0$.

In the case of $V$, the proof of Lemma~\ref{Lem:lift2} shows
that the action of $G=\mathbb{Z}_{2r}$ on $SE(\lambda_1)^{2r}_{r,p,r}$ given by
 \begin{equation}
 \label{nr1b}
  [X:Y:Z]\mapsto [X: \, \rho^p_{2r} Y:Z] = [\, \rho^{-p}_{2r} X:  Y: \, \rho^{-p}_{2r} Z] \;,
\end{equation}
can be lifted to an action on its normalization. In the local coordinates it is given by
\begin{equation}
\label{LocalActionG}
 z_0, z_1, z_\lambda, z_\infty \mapsto  \rho_{2r} z_0, \rho_{2r}^{2r-1}z_1=\rho_{2r}^{-1}z_1, \rho_{2r}^{2r-1}z_\lambda=\rho_{2r}^{-1}z_\lambda,  \rho_{2r} z_\infty \;.
\end{equation}
The second affine model $SE(\lambda_2)^{2r}_{r,2r-p,p}$ is obtained by sending $p \to 2r-p$, but at the same time
we are also considering the lift of the different action
 \begin{equation}
 \label{nr2b}
   [X:Y:Z]\mapsto [X: \, \rho^{-p}_{2r} Y:Z] = [\, \rho^{p}_{2r} X:  Y: \, \rho^{p}_{2r} Z] 
\end{equation}
because the actions in Equations~(\ref{nr1b}) and~(\ref{nr2b}) combine to give the action~(\ref{involution2}).
One then checks that action of $G$ on the local coordinates of the normalization of $SE(\lambda_2)^{2r}_{r,2r-p,p}$ is the same as the one in Equation~(\ref{LocalActionG}).
It follows that a neighborhood of each singular point is equivalent to the quotient of $\mathbb{C}^2$ by the action of
a linear diagonal automorphism with eigenvalues $\lbrace \exp{(\frac{2\pi i a}{2r})}, \exp{(\frac{2\pi i b}{2r})}\rbrace$ with $a,b\in \lbrace -1, 1 \rbrace$.
If $a\not =b$ the singularity is a rational double point of type $A_{2r-1}$.
\end{proof}
We will now consider the minimal resolution $S_0 \to V_0$ of the rational singularities of the product-quotient surface $V_0 = (C_1 \times C_2)/G_0$ 
where $C_1$ and $C_2$  are the smooth irreducible curves of genus $2r-1$ and the diagonal action with fixed points of the group $G_0=\mathbb{Z}_{2}$.
In the special case $q=r$, we will also consider the minimal resolution $S \to V$ of the rational singularities of the product-quotient surface $V = (C_1 \times C_2)/G$ 
where the diagonal action with fixed points of the group $G=\mathbb{Z}_{2r}$ is given by Equation~(\ref{involution2}).
\begin{remark}
More details on product-quotient surfaces and a summary on how cyclic quotient singularities can be resolved
by so-called Hirzebruch-Jung strings can be found in \cite[Sec.~2]{MR2956256}.
\end{remark}
Following \cite[Rem.~2]{MR2956256} we denote by $K_{V_0}$  the canonical Weil divisor on the normal surface corresponding to the inclusion
of the smooth locus of $V_0$. According to Mumford we have an intersection product with values in $\mathbb{Q}$ for Weil divisors on a normal surface
such that
\begin{equation}
 K_{V_0}^2 = \dfrac{8 \, (g(C_1)-1) \, (g(C_2)-1)}{|G_0|} = 16 \, (r-1)^2 \;.
\end{equation}
The following lemma follows;
\begin{lemma}
\label{Lem:Inv}
The surface $S_0$ has the following invariants:
\begin{equation}
\label{KS2}
 K_{S_0}^2  = 16 \, (r-1)^2 , \quad e(S_0) = 24 + 8 \, (r-1)^2 , \quad \chi(\mathcal{O}_{S_0}) = 2 + 2 \, (r-1)^2, \quad \tau(S_0)=-16\;.
\end{equation}
\end{lemma}
\begin{proof}
The singular surface $V_0$ has only rational double point singularities. Therefore, $K_{V_0}^2$ is an integer and agrees with $K_{S_0}^2$ \cite{MR927963}.
Using \cite[Cor.~2.7]{MR2956256} it follows that
$$
 e(S_0) = \frac{K_{V_0}^2}{2} + \frac{3}{2} \, k
$$
where $k=16$ is the number of ordinary double points. The holomorphic Euler characteristic $\chi(\mathcal{O}_{S_0})$ is then obtained by Noether's formula. The signature is
computed by the Thom-Hirzebruch signature index
$$
 \tau(S_0) = \frac{K_{S_0}^2-2e(S_0)}{3} = -k = -16 \;. 
$$
\end{proof}
\begin{remark}
The same computation cannot be used to obtain the self-intersection of the canonical divisor on the surface $S$ as not all of the singularities of the product-quotient surface $V$
are rational double point singularities. In fact, we have that
\begin{equation}
 K_{V}^2 = \dfrac{8 \, (g(C_1)-1) \, (g(C_2)-1)}{|G_0|} = \frac{16 \, (r-1)^2}{r} \;
\end{equation}
is an integer only for $r=1$.
\end{remark}
We have the following lemma:
\begin{lemma}
\label{Lem:Pi1}
It follows that
\begin{equation}
\label{fundamental_group}
  \pi_1(S_0) = \pi_1(V_0) \;, \qquad  \pi_1(S) = \pi_1(V) \;.
\end{equation}
In particular, the surfaces $S_0$ and $S$ have irregularities $q(S_0)=4(r-1)$ and $q(S)=0$, respectively.
\end{lemma}
\begin{proof}
Since the minimal resolution $S_0 \to V_0$ and $S \to V$ of the singularities of $V_0$ and $V$, respectively, replaces each singular point by a tree of smooth rational curves, 
the claim about the fundamental group then follows by van Kampen's theorem.  Notice that by construction we have $C_1/G=C_2/G=\mathbb{P}^1$ in the special case $q=r$. 
By \cite[Prop.~3.15]{MR1737256} it follows that the surface $S$ falls into the $D^0$-case, hence $q(S)=0$.

Let us now compute the irregularity $q(S_0)$.
We obtained $S_0$  from the singular quotient variety $V_0=T/G_0$ by blowing up $16$ ordinary double point singularities where $T= C_1 \times C_2$ and the action of $G_0$ 
was given in Equation~(\ref{involution}) and could be lifted to $T$. Alternatively, we can blow up the $16$ fixed points of $G_0$ on $T$ 
first to construct a surface $\tilde{T}$ whose projection $\pi$ onto $\tilde{T}/G_0$ is $S_0$. 
The latter point of view makes it easier to construct a rational basis of the first homology of $S_0$ as the dimension of the first homology does not change during the process of blowing-up. 
We denote the two bases for the first homology of the singular plane curves $SE(\lambda_1)^{2r}_{r,p,q}$ and $SE(\lambda_2)^{2r}_{r,2r-p,2r-q}$, respectively, 
constructed in  Section~\ref{SECperiods} by 
$$\left(\mathfrak{a}^{(k)}_i,\mathfrak{b}_j^{(k)}\right)_{i,j=1}^{2r-1} \quad \text{for $k=1$ and $k=2$}.$$
These cycles lift to $4g=4(2r-1)$ non-trivial one-cycles $(\gamma_m)_{m=1}^{4(2r-1)}$ on $\tilde{T}$. The involution~(\ref{involution}) identifies the non-trivial one-cycles 
 $$ \mathfrak{a}^{(1)}_i \times \infty \quad \text{and} \quad\infty \times \mathfrak{a}^{(2)}_i$$
 with the one-cycles
  $$\mathfrak{a}^{(1)}_{i+r} \times \infty \quad \text{and} \quad \infty \times \mathfrak{a}^{(2)}_{i+r},$$
 respectively, for $i=1,\dots,r-1$. An analogue statement holds for the $B$-cycles. Moreover, the one-cycles
 $$\mathfrak{a}^{(1)}_r \times \infty \quad \text{and} \quad \infty \times \mathfrak{a}^{(2)}_r$$
become homotopic to zero on $S_0$ when the $y$-sheet is identified with the $(-y)$-sheet. 
 Therefore, the $4(r-1)$ one-cycles obtained from the projections of the remaining one-cycles
$$
 \pi_*\big(\mathfrak{a}^{(1)}_{i} \times \infty\Big),  \quad  \pi_*\big(\infty \times \mathfrak{a}^{(2)}_{i}\big), \quad \pi_*\big(\mathfrak{b}^{(1)}_{i} \times \infty\Big), \quad 
 \pi_*\big(\infty \times \mathfrak{b}^{(2)}_{i}\big) 
$$
for $i=1,\dots,r-1$ form a basis of the first homology of $S_0$ over $\mathbb{Q}$.
\end{proof}
We can now compute the Hodge-diamond of the surface $S_0$.
\begin{corollary}
The Hodge-diamond of the surface $S_0$ is given by
$$
\begin{array}{ccccc}
                 & 	       &1                    &            & \\
                 & q(S_0) &                      & q(S_0) & \\
p_g(S_0)  &             & h^{1,1}(S_0) &             & p_g(S_0) \\
                 & q(S_0) &                      & q(S_0) & \\
                 & 	       &1                    &            & 
\end{array}
$$
with $q(S_0)=4(r-1)$, $p_g(S_0)=1 + 2\, (r^2-1)$, and $h^{1,1}(S_0) = 20 + 2 \, (r^2-1)$.
\end{corollary}
\begin{proof}
Lemmas~\ref{Lem:Inv} and~\ref{Lem:Pi1} imply that
$$
  2 + 2 \, (r-1)^2  = \chi(\mathcal{O}_{S_0}) = 1 - q(S_0) + p_g(S_0) = 1 - 4\, (r-1) + p_g(S_0) \;,
$$
hence $p_g(S_0)=1 + 2\, (r^2-1)$.  The Hodge index theorem implies that the intersection form on $\mathrm{H}^{1,1}(S_0) \cap \mathrm{H}^2(S_0,\mathbb{R})$
has one positive eigenvalue and $h^{1,1}(S_0)-1$ negative eigenvalues. Therefore, we have
$$
 -16 = \tau(S_0)  = 2 \, p_g(S_0) + 1 - \big( h^{1,1}(S_0) -1 \big) = 4 + 2 \, (r^2-1) - h^{1,1}(S_0) \;,
$$
hence $h^{1,1}(S_0) = 20 + 2 \, (r^2-1)$.
\end{proof}

\begin{definition}
We call the minimal resolution $\rho: S_0 \to V_0$ of the 16 rational double point singularities on the product-quotient surface $V_0 = (C_1 \times C_2)/G_0$ 
a generalized Kummer surface where $C_1$ and $C_2$  are the smooth irreducible curves of genus $2r-1$ obtained as minimal resolution
of the two singular curves $SE(\lambda_1)^{2r}_{r,p,q}$ and $SE(\lambda_2)^{2r}_{r,2r-p,2r-q}$,
and the diagonal action with fixed points of $G_0=\mathbb{Z}_{2}$ is given by Equation~(\ref{involution}).
We will always assume that the modular parameters $\lambda_1$ and $\lambda_2$ are generic, i.e., not equal to $0, 1, \infty$ and 
that the integers $p, q, r$ satisfy the inequalities~(\ref{range}) and divisibility constraints~(\ref{divisibility}).
\end{definition}

\begin{remark} If it is necessary we will also write $S_0^{(r,p,q)}(\lambda_1, \lambda_2)$ and $V_0^{(r,p,q)}(\lambda_1, \lambda_2)$.
Note that $V_0^{(r,p,q)} \cong V_0^{(r,2r-p,2r-q)}$ and $S_0^{(r,p,q)} \cong S_0^{(r,2r-p,2r-q)}$ since mapping $(p,q) \mapsto (2r-p,2r-q)$ amounts to interchanging the roles 
of $(\zeta_1, \eta_1, \lambda_1)$ with $(\zeta_2, \eta_2, \lambda_2)$, but leaves the total space invariant.
\end{remark}

In the proof of Lemma~\ref{Lem:Pi1} we constructed a rational basis of $4g=4(2r-1)$ non-trivial one-cycles $(\gamma_m)_{m=1}^{4(2r-1)}$ on $\tilde{T}$
where $\tilde{T}$ was obtained by blowing up the 16 fixed points on $C_1 \times C_2$ first such that $S_0$ is obtained as projection $\pi: \tilde{T} \to S_0=\tilde{T}/G_0$.
We now look at the second homology of $S_0$. We find (1) two-cycles from the exceptional divisors generated during the blow-up constructing $S_0$, and (2)
two-cycles $(\sigma_{m,n})_{m,n=1}^{4(2r-1)}$ with $m \not = n$ from the projection of two-cycles on $\tilde{T}$, i.e.,
$\sigma_{m,n} = \pi_*(\gamma_m \times \gamma_n)$. 

Using Lemma~\ref{Lem:1-form} the two-form $d\zeta_1/\eta_1 \boxtimes d\zeta_2/\eta_2$ defines a holomorphic differential two-form 
on $T=C_1 \times C_2$, hence on $\tilde{T}$, that induces a unique holomorphic two-form $\omega$ on $S_0^{(r,p,q)}$ such that
\begin{equation}
\label{HolomorphicForm}
\oiint_{ \sigma_{m,n}} \omega =  \iint_{ \gamma_m \times \gamma_n} \frac{d\zeta_1}{\eta_1} \boxtimes \frac{d\zeta_2}{\eta_2} \;. 
\end{equation}
We have the following lemma:
\begin{lemma}
\label{Lem:KummerPeriods}
The periods of $\omega$ over the two-cycles $\sigma_{m,n}$ on $S_0^{(r,p,q)}(\lambda_1, \lambda_2)$ are given by
\begin{equation}
\label{SEKummerPeriods}
\begin{split}
 F^{(k,l)}_{r,p,q}(\lambda_1,\lambda_2)_1 &=   \oiint_{\pi_*\left(\mathfrak{a}^{(1)}_{k} \times \mathfrak{a}^{(2)}_{l} \right)} \omega = f^{(k)}_{r,p,q}(\lambda_1) \, f^{(l)}_{r,2r-p,2r-q}(\lambda_2)\;,\\
 F^{(k,l)}_{r,p,q}(\lambda_1,\lambda_2)_2 &=   \oiint_{\pi_*\left(\mathfrak{a}^{(1)}_{k} \times \mathfrak{b}^{(2)}_{l} \right)}  \omega = f^{(k)}_{r,p,q}(\lambda_1) \, f^{(l) \, \prime}_{r,2r-p,2r-q}(\lambda_2)\;,\\
 F^{(k,l)}_{r,p,q}(\lambda_1,\lambda_2)_3 &=   \oiint_{\pi_*\left(\mathfrak{b}^{(1)}_{k} \times \mathfrak{a}^{(2)}_{l} \right)}  \omega = f^{(k) \, \prime}_{r,p,q}(\lambda_1) \, f^{(l)}_{r,2r-p,2r-q}(\lambda_2)\;,\\
 F^{(k,l)}_{r,p,q}(\lambda_1,\lambda_2)_4 &=   \oiint_{\pi_*\left(\mathfrak{b}^{(1)}_{k} \times \mathfrak{b}^{(2)}_{l} \right)}  \omega = f^{(k) \, \prime}_{r,p,q}(\lambda_1) \, f^{(l) \, \prime}_{r,2r-p,2r-q}(\lambda_2)\;,
 \end{split}
\end{equation}
where $1\le k, l \le r-1$, and $f_{r,p,q}(\lambda)$ and $f'_{r,p,q}(\lambda)$ were computed in Lemma~\ref{SuperellipticPeriods}. The periods in Equations~(\ref{SEKummerPeriods}) satisfy the quadratic relation
$$
F^{(k,l)}_{r,p,q}(\lambda_1,\lambda_2)_1 \; F^{(k,l)}_{r,p,q}(\lambda_1,\lambda_2)_4 - F^{(k,l)}_{r,p,q}(\lambda_1,\lambda_2)_2 \; F^{(k,l)}_{r,p,q}(\lambda_1,\lambda_2)_3 =0 \;.
$$
\end{lemma}
\begin{remark}
Whenever $k=1$, i.e., we are looking at the principal branch, we will drop the superscript and write $F_{r,p,q}(\lambda_1), F'_{r,p,q}(\lambda_1)$, etc.
\end{remark}
\begin{proof}
First, we observe that on $S_0^{(r,p,q)}$ the two-cycles $\sigma_{m,n}$ are transcendental only if they are among the $4(r-1)^2$ two-cycles used in Equations~(\ref{SEKummerPeriods}).
The lemma and Equations~(\ref{SEKummerPeriods}) then follow from Lemma~\ref{SuperellipticPeriods}.
The quadratic period relation is obvious from Equations~(\ref{SEKummerPeriods}).
\end{proof}
\begin{remark}
As special functions in the moduli $\lambda_1, \lambda_2$, the periods $F_{r,p,q}(\lambda_1,\lambda_2)_l$ for $l=1,\dots,4$
satisfy (up to an overall renormalizing function in $\lambda_1, \lambda_2$) the rank-four system described in Proposition~\ref{LemmaTensorSystem} that is the (outer) tensor product of two hypergeometric functions.
\end{remark}

\subsection{Fibration structures}
Kuwata and Shioda determined explicit formulas determining the elliptic parameter denoted by 
$U$ and the elliptic fiber coordinates $X, Y$ for each inequivalent Jacobian elliptic fibration 
labeled $(\mathfrak{J}_i)_{i=1}^{11}$ in \cite{MR2409557} on the Kummer surface $S_0$  that is
the minimal nonsingular model for the quotient of a product of two non-isogenous elliptic curves
by the involution automorphism. That is, the Kummer surface is a Jacobian elliptic $K3$ surface of Picard-rank $18$ with function field
$$
  \mathbb{C}(S_0) = \mathbb{C}(U,X,Y) \;
$$
with one relation between $U, X, Y$ given in form of a Weierstrass equation.
 We checked that for all of their 11 elliptic fibrations $(\mathfrak{J}_i)_{i=1}^{11}$ we have
\begin{equation}
\label{relation_2form}
 \pi^* \left(dU \wedge \frac{dX}{Y}\right) =   \frac{d\zeta_1}{\eta_1} \boxtimes \frac{d\zeta_2}{\eta_2}  \;
\end{equation}
where $\pi$ is a projection map analogous to the one used in the proof of Lemma~\ref{Lem:Pi1}, and $(\zeta_1, \eta_1)$ and $(\zeta_2, \eta_2)$ are the 
variables used to define two elliptic curves in Legendre normal form.
In this article, we have replaced the elliptic curves in the classical Kummer construction by the two superelliptic curves in Equations~(\ref{ReminderSEs}) of genus $2r-1$. 
For general values for $r, p, q$, we still consider $\zeta_1$ and $\zeta_2$  as independent variables, 
and $\eta_1$ and $\eta_2$ as variables dependent on $r, p, q$ and $\zeta_1$ or $\zeta_2$, respectively,  by means of Equations~(\ref{ReminderSEs}).  Thus, the definition for the functions $U$ and $X$ 
given in \cite{MR2409557} then cannot change as we vary $r, p, q$ if we want to maintain the differential relation~(\ref{relation_2form}). However, we can read off 
from \cite[Table~1]{MR2409557} that only for the cases $(\mathfrak{J}_i)_{i=4}^{11}$ are the variables $U, X$ contained in $\mathbb{C}(\zeta_1, \zeta_2)$. Among those, only the 
five remaining cases $(\mathfrak{J}_i)_{i=4}^{8}$ have a two-torsion section contained in their Mordell-Weil group which is necessary to generalize a fibration of elliptic curves to a fibration of superelliptic curves
 in the sense of Definition~\ref{WhatIsSuperelliptic}. We will now derive the equations generalizing fibrations $(\mathfrak{J}_i)_{i=4}^{8}$ on $V_0$.  As we will see, for fibrations $\mathfrak{J}_5$ and $\mathfrak{J}_8$
 we have to specialize to the case $q=3r-p$ (cf.~Remark~\ref{simplify}).

The construction for the generalized Kummer variety is slightly more complicated than in the classical case:  
we started with two singular curves $SE(\lambda_1)^{2r}_{r,p,q}$ and $SE(\lambda_2)^{2r}_{r,2r-p,2r-q}$ 
with affine coordinates $(\zeta_1, \eta_1)$ and  $(\zeta_2, \eta_2)$, respectively. We constructed their minimal resolutions $C_1$ and $C_2$
with normalization maps $\psi_1$ and $\psi_2$, respectively, by Lemma~\ref{Lem:lift} and lifted the involution automorphism~(\ref{involution}) to $T=C_1 \times C_2$. 
The coordinates $\lbrace \zeta_1, \eta_1, \zeta_2, \eta_2\rbrace $ 
when composed with $\psi_1$ and $\psi_2$ then are rational functions on $T$, hence on 
the space $\tilde{T}$ obtained by blowing up the fixed points of $G_0$. Therefore, expressions in these coordinates invariant under the involution descend to rational functions on $S_0$, and the fibrations $(\mathfrak{J}_i)_{i=4}^{8}$ become algebraic relations between these rational functions.
 
\begin{remark}
 Note that we will be using the square roots $\Lambda_1$ and $\Lambda_2$ of the modular parameters $\lambda_1$ and $\lambda_2$ for the two superelliptic curves in Equations~(\ref{ReminderSEs}).
 We have already seen in Theorem~\ref{thm1} that this allows us to find \emph{rational} relations to the variables $A, B$ that will appear as variables in Appell's hypergeometric series $F_2$.
\end{remark} 

\subsubsection{Generalization of fibration $\mathfrak{J}_4$}
\label{IsoFibJ4}
We set $Y= \eta_1 \, \eta_2$ and $U=\zeta_1, X=\zeta_2$ or  $\tilde{U}=\zeta_2, \tilde{X}=\zeta_1$.
We have the following lemma:
\begin{lemma}
The functions $U, X, Y$ or $\tilde{U}, \tilde{X}, Y$ descend to rational functions on the generalized Kummer variety $S_0^{(r,p,q)}(\Lambda_1^2,\Lambda_2^2)$ satisfying the relation
\begin{equation}
\label{IsoFibration_1}
 Y^{2r}  =  U^{p+q-r} \; (U-1)^{2r-p} \; (U-\Lambda_1^2)^{2r-p}  \,  X^{3r-p-q} \; (X-1)^{p} \; (X-\Lambda_2^2)^{p}
\end{equation}
or the relation
\begin{equation}
\label{IsoFibration_2}
 Y^{2r}  =  \tilde{U}^{3r-p-q} \; (\tilde{U}-1)^{p} \; (\tilde{U}-\Lambda_2^2)^{p}  \,  \tilde{X}^{p+q-r} \; (\tilde{X}-1)^{2r-p} \; (\tilde{X}-\Lambda_1^2)^{2r-p} \;,
\end{equation}
respectively, and
\begin{equation}
  \pi^* \left(dU \wedge \frac{dX}{Y}\right) =  \frac{d\zeta_1}{\eta_1} \boxtimes \frac{d\zeta_2}{\eta_2} \;, \qquad
   \pi^* \left(d\tilde{U} \wedge \frac{d\tilde{X}}{\tilde{Y}}\right) =  \frac{d\zeta_1}{\eta_1} \boxtimes \frac{d\zeta_2}{\eta_2}    \;.
\end{equation}
\end{lemma}
\begin{remark}
For $r=p=q=1$, Equations~(\ref{IsoFibration_1}) and (\ref{IsoFibration_2}) coincide with the Jacobian elliptic fibration $\mathfrak{J}_4$ on a Kummer
surface of two non-isogenous elliptic curves established in \cite{MR2409557}.
Equations~(\ref{IsoFibration_1}) and (\ref{IsoFibration_2}) then describe the so-called first and second Kummer pencil.
\end{remark}

\subsubsection{Generalization of fibration $\mathfrak{J}_5$}
We define the fibration parameter
\begin{equation}
\label{fibrationJ5param}
 U =  \frac{\left( \zeta_1 - \zeta_2\right) \, \left( \Lambda_2^2  \, \zeta_1 +(\Lambda_1^2-1) \, \zeta_2 - \Lambda_1^2 \, \Lambda_2^2\right)}
 {\left(\Lambda_2^2 \, \zeta_1 - \zeta_2\right) \, \left(\zeta_1  + (\Lambda_1^2-1) \, \zeta_2 - \Lambda_1^2\right)} \;,
\end{equation}
and the coordinates
\begin{equation}
\label{fibrationJ5coords}
 \begin{split}
  X & =  \frac{R \, \left(\zeta_1 - \zeta_2\right)  \, \left(\zeta_1 - \Lambda_1^2\right)}{\zeta_1 \, \left(\zeta_1 + (\Lambda_1^2-1) \, \zeta_2 - \Lambda_1^2\right)} \;,\\
  Y & =  \frac{R \, \Lambda_1^2 \, \left(1-\Lambda_1^2\right) \,  (U-1) \, \zeta_2 \, (\zeta_1 - \zeta_2) \, \eta_1 \, \eta_2}{\zeta_1^2 \, \left(\Lambda_2^2 \, \zeta_1 - \zeta_2\right) \, \left(\zeta_1 + (\Lambda_1^2-1) \, \zeta_2 - \Lambda_1^2\right)^2} \;,\\
  R & = \Lambda_1^2 \, \left(\Lambda_2^2-1\right) \, (U-1) \, \left( U -1 + \Lambda_1^2 \, (1-\Lambda_2^2) \right) \\
  & \quad \times \; \left( (\Lambda_1^2 \, \Lambda_2^2 - \Lambda_1^2 - \Lambda_2^2) \, U + \Lambda_2^2 \right) \;.
   \end{split}
\end{equation} 
We have the following lemma:
\begin{lemma}
The functions $U, X, Y$ defined in Equations~(\ref{fibrationJ6param}) and (\ref{fibrationJ6coords}) descend to rational functions on $S_0^{(r,p,q=3r-2p)}(\Lambda_1^2,\Lambda_2^2)$ satisfying the relation
\begin{equation}
\label{fibrationJ5fib}
\begin{split}
 Y^{2r} = \Lambda_1^{-4p} \, \left(1-\Lambda_1^2 \right)^{2(r-p)} \, \left(1-\Lambda_2^2\right)^{-2r} \, \left(U-1\right)^{4(r-p)} \, X^{2r-p}  \qquad \qquad\\
\times \left( X + \Lambda_1^2 \, (\Lambda_2^2-1) \, (U-1) \, \left( (\Lambda_1^2 \, \Lambda_2^2 -1) \, U - \Lambda_1^2+1 \right) \, \left( (\Lambda_1^2 \, \Lambda_2^2-\Lambda_1^2-\Lambda_2^2)\, U + \Lambda_2^2\right) \right)^p \\
\times \left( X + \Lambda_1^2 \, (\Lambda_2^2-1) \, U \, (U-1) \, \left(U-1 + \Lambda_1^2 \, (1-\Lambda_2^2) \right) \, \left( (1 - \Lambda_1^2) \, \Lambda_2^2 \, U +\Lambda_1^2 - \Lambda_2^2\right) \right)^p 
\end{split}
\end{equation}
and
\begin{equation}
  \pi^* \left(dU \wedge \frac{dX}{Y}\right) =   \frac{d\zeta_1}{\eta_1} \boxtimes \frac{d\zeta_2}{\eta_2}  \;.
\end{equation}
\end{lemma}
\begin{proof}
By direct computation using formulas~(\ref{fibrationJ5param}) and (\ref{fibrationJ5coords}).
\end{proof}
\begin{remark}
For $r=p=q=1$, Equation~(\ref{fibrationJ5fib}) coincides with the Jacobian elliptic fibration $\mathfrak{J}_5$ on a Kummer
surface of two non-isogenous elliptic curves established in \cite{MR2409557}.
\end{remark}

\subsubsection{Generalization of fibration $\mathfrak{J}_6$}
\label{TwistFibJ6}
We define the fibration parameter
\begin{equation}
\label{fibrationJ6param}
 U = \frac{\zeta_1}{\zeta_2}
\end{equation}
and the coordinates
\begin{equation}
\label{fibrationJ6coords}
 \begin{split}
  X & = \frac {\zeta_1 \,  \left(\zeta_1 - \Lambda_1^2\right)  \,  \left( \zeta_1-\zeta_2 \right) \,   \left( \Lambda_2^2 \zeta_1-\zeta_2 \right) }{\zeta_2^3 \, \left( \zeta_1-1 \right) } \;,\\
  Y & = \frac { \left( \Lambda_1^2-1\right) \, \zeta_1^{2} \left( \zeta_1-\zeta_2\right)   \left( \Lambda_2^2 \,\zeta_1  -\zeta_2 \right) \, \eta_1\, \eta_2}{\zeta_2^5 \,\left( \zeta_1-1 \right)^2} \;.
   \end{split}
\end{equation} 
We have the following lemma:
\begin{lemma}
The functions $U, X, Y$ defined in Equations~(\ref{fibrationJ6param}) and (\ref{fibrationJ6coords}) descend to rational functions on $S_0^{(r,p,q)}(\Lambda_1^2,\Lambda_2^2)$ satisfying the relation
\begin{equation}
\label{fibrationJ6fib}
\begin{split}
 Y^{2r} = \left(1 - \Lambda_1^2\right)^{2(r-p)} \, U^{-2p+q+r} \,  X^{2r-p} \qquad \qquad \qquad \\
 \times \, \left( X - U \, (U-\Lambda_1^2) \, (\Lambda_2^2 \, U -1) \right)^p \,  \left( X - U \, (U-1) \, (\Lambda_2^2 \, U -\Lambda_1^2) \right)^p
\end{split}
\end{equation}
and
\begin{equation}
  \pi^* \left(dU \wedge \frac{dX}{Y}\right) =  \frac{d\zeta_1}{\eta_1} \boxtimes \frac{d\zeta_2}{\eta_2}  \;.
\end{equation}
\end{lemma}
\begin{proof}
By direct computation using formulas~(\ref{fibrationJ6param}) and (\ref{fibrationJ6coords}).
\end{proof}

\begin{remark}
For $r=p=q=1$, Equation~(\ref{fibrationJ6fib}) coincides with the Jacobian elliptic fibration $\mathfrak{J}_6$ on a Kummer
surface of two non-isogenous elliptic curves established in \cite{MR2409557}.
\end{remark}

\subsubsection{Generalization of fibration $\mathfrak{J}_7$}
\label{TwistFibJ7}
We define the fibration parameter
\begin{equation}
\label{fibrationJ7param}
 U = \frac{(\zeta_2 - \Lambda_2^2) \, (\zeta_1 - \zeta_2)}{(\zeta_2 -1) \, (\Lambda^2_2 \, \zeta_1 - \zeta_2)}
\end{equation}
and the coordinates
\begin{equation}
\label{fibrationJ7coords}
 \begin{split}
  X & = \frac{\Lambda^2_2 \, (\Lambda^2_2-1)^2 \, \zeta_1 \zeta_2 \, (\zeta_1-1)^2 \, (\zeta_2 - \Lambda^2_2) \, (\zeta_1 - \zeta_2)}{(\zeta_2 - 1)^3 \, (\Lambda_2^2 \, \zeta_1 - \zeta_2)^3} \;,\\
  Y & = -\frac{\Lambda_2^{2} \, (\Lambda^2_2-1)^{3} \, (\zeta_1-1)^2 \, \zeta_2 \, (\zeta_2-\Lambda_2^2) \, (\zeta_1-\zeta_2)^2 \, \eta_1 \, \eta_2}{(\Lambda^2_2 \, \zeta_1 - \zeta_2)^4 \, (\zeta_2-1)^5}  \;.
 \end{split}
\end{equation} 
We have the following lemma:
\begin{lemma}
The functions $U, X, Y$ defined in Equations~(\ref{fibrationJ7param}) and (\ref{fibrationJ7coords}) descend to rational functions on $S_0^{(r,p,q)}(\Lambda_1^2,\Lambda_2^2)$ satisfying the relation
\begin{equation}
\label{fibrationJ7fib}
\begin{split}
 Y^{2r} = \Lambda_2^{2(r-q)} \, \left(\Lambda_2^2-1\right)^{2(p-r)} \, U^{2p-q-r} \, (U-1)^{2(r-q)} \, X^{p+q-r} \qquad  \\
 \times \; \left( X^2 - U \, (U-1) \, \Big(\left(\Lambda_1^2 \, \Lambda_2^2+1\right) \, U - \Lambda_1^2 - \Lambda_2^2 \Big) \, X + \Lambda_1^2 \, \Lambda_2^2 \, U^2 \, (U-1)^4 \right)^{2r-p}
\end{split}
\end{equation}
and
\begin{equation}
\label{relationJ72form}
  \pi^* \left(dU \wedge \frac{dX}{Y}\right) =  \frac{d\zeta_1}{\eta_1} \boxtimes \frac{d\zeta_2}{\eta_2}  \;.
\end{equation}
\end{lemma}
\begin{proof}
By direct computation using formulas~(\ref{fibrationJ7param}) and (\ref{fibrationJ7coords}).
\end{proof}
\begin{remark}
For $r=p=q=1$, Equation~(\ref{fibrationJ7fib}) coincides with the Jacobian elliptic fibration $\mathfrak{J}_7$ on a Kummer
surface of two non-isogenous elliptic curves established in \cite{MR2409557}.
\end{remark}

\subsubsection{Generalization of fibration $\mathfrak{J}_8$}
We define the fibration parameter
\begin{equation}
\label{fibrationJ8param}
 U =  - \frac{\left( \zeta_1 - \zeta_2\right) \, \left( \zeta_2 - \Lambda_2^2\right)}
 {\Lambda_2^2 \, \left(\Lambda_2^2-1\right) \, \zeta_1 \, (\zeta_1-1)} \;,
\end{equation}
and the coordinates
\begin{equation}
\label{fibrationJ8coords}
 \begin{split}
  X & =\frac{U \, \left( (\Lambda_1^2-1) \, (\Lambda_2^2-1) \, U-1\right) \, (\zeta_2-1) \, \left(\Lambda_2^2\, \zeta_1 -\zeta_2\right)}{\left(\Lambda_2^2-1\right) \, \zeta_2 \, (\zeta_1-1)}\;,\\
  Y & =   \frac{ \Lambda_2^2 \, U^3 \, \left( (\Lambda_1^2-1) \, (\Lambda_2^2-1) \, U-1\right) \, \left(\Lambda_2^2 \, \zeta_1- \zeta\right) \, \eta_1 \, \eta_2}{\zeta_2^2 \, (\zeta_1-1) \, \left(\zeta_2- \Lambda_2^2\right)} \;.
\end{split}
\end{equation} 
We have the following lemma:
\begin{lemma}
The functions $U, X, Y$ defined in Equations~(\ref{fibrationJ8param}) and (\ref{fibrationJ8coords}) descend to rational functions on $S_0^{(r,p,q=3r-2p)}(\Lambda_1^2,\Lambda_2^2)$ satisfying the relation
\begin{equation}
\label{fibrationJ8fib}
\begin{split}
 Y^{2r} = & \; \Lambda_2^{4(p-r)} \, \left(1-\Lambda_1^2 \right)^{2(p-r)} \, U^{4(p-r)} \, X^{p}  \qquad \qquad \qquad\\
\times&  \left( X^2 - U \, \left((2 \, \Lambda_1^2 \, \Lambda_2^2- \Lambda_1^2 - \Lambda_2^2+2) \, U -2 \right) \, X \right. \\
& \left. - U^2 \, (U-1) \, \left(\Lambda_1^2 \, \Lambda_2^2 \, U -1 \right) \, \left((\Lambda_1^2-1) \, (\Lambda_2^2-1) U -1 \right)\right)^{2r-p}
\end{split}
\end{equation}
and
\begin{equation}
  \pi^* \left(dU \wedge \frac{dX}{Y}\right) =  \frac{d\zeta_1}{\eta_1} \boxtimes \frac{d\zeta_2}{\eta_2}  \;.
\end{equation}
\end{lemma}
\begin{proof}
By direct computation using formulas~(\ref{fibrationJ8param}) and (\ref{fibrationJ8coords}).
\end{proof}
\begin{remark}
For $r=p=q=1$, Equation~(\ref{fibrationJ8fib}) coincides with the Jacobian elliptic fibration $\mathfrak{J}_8$ on a Kummer
surface of two non-isogenous elliptic curves established in \cite{MR2409557}.
\end{remark}

\subsection{Fibrations by twists and base transformations}

We now investigate the generalizations of fibrations $\mathfrak{J}_4$, $\mathfrak{J}_6$, and $\mathfrak{J}_7$ given by
Equations~$(\ref{IsoFibration_1})$, $(\ref{IsoFibration_2})$, $(\ref{fibrationJ6fib})$, and $(\ref{fibrationJ7fib})$ in more detail.
First, we will establish in what sense Equations~$(\ref{IsoFibration_1})$, $(\ref{IsoFibration_2})$ and $(\ref{fibrationJ7fib})$
define actual fibrations with fibers of genus $2r-1$ on the generalized Kummer variety $S_0$.
\begin{proposition}
\label{IsoFibration}
On $S_0^{(r,p,q)}$ there are two isotrivial fibrations $f'_1:S_0 \to C'_1$ and $f'_2:S_0 \to C'_2$ where $C'_1$ and $C'_2$ are smooth irreducible curves of genus $r-1$
and the fibers are isomorphic to the curves $C_2$ and $C_1$, respectively, of genus $2r-1$. 
\end{proposition}

\begin{remark}
As pointed out in \cite[Rem.~2.5]{MR1737256}, in the context of isotrivial fibrations -- like the ones in Lemma~\ref{IsoFibration} -- a converse to Arakelov's theorem does not need to hold, i.e., we have
$$
 K_{S_0}^2 \not = 8 \, \Big( g(\,\text{fiber}) - 1\Big) \,  \Big( g(\,\text{base}) - 1\Big) = 32 \, r \, (r-1)\;.
$$
In fact a small deformation of a constant moduli fibration does not need to be again a constant moduli fibration.
\end{remark}

We first proof a lemma construction the quotient curve $C_1/G_0$. In Lemma~\ref{projective_model} we constructed the non-singular irreducible curve $C_1$ of genus $2r-1$ as minimal resolution of 
the singular curve $SE(\Lambda_1^2)^{2r}_{r,p,q}$ which we denoted by $\psi_1: C_1\to SE(\Lambda_1^2)^{2r}_{r,p,q}$,  and also the degree-$2r$ map $[X:Y:Z] \mapsto [X:Z]$ denoted by 
$\phi_1:SE(\Lambda_1^2)^{2r}_{r,p,q} \to \mathbb{P}^1$. Furthermore, in Lemma~\ref{Lem:lift} we proved that the involution automorphism~(\ref{involution_EC}) of the action of $G_0=\mathbb{Z}_2$ 
on $SE(\Lambda_1^2)^{2r}_{r,p,q}$ lifts to an automorphism on the minimal resolution $C_1$. We first prove the following lemma:
\begin{lemma}
The quotient curve $C_1/G_0$ is isomorphic to the minimal resolution $C_1'$ of the singular curve $SE(\Lambda_1^2)^{r}_{r,p,q}$, i.e.,
\begin{equation}
 C_1/G_0 = C_1' \overset{\psi_1'\;}{\dashrightarrow} SE(\Lambda_1^2)^{r}_{r,p,q} \;,
\end{equation}
and is a smooth irreducible curve of genus $r-1$.
\end{lemma}
\begin{proof}
We first observe that on the curve $SE(\lambda)^{2r}_{r,p,q}$ the quotient by $G_0=\mathbb{Z}_2$ can be constructed by using the rational transformation 
\begin{equation}
\label{transfo_quotient}
 \tilde{X}=X\,, \quad \tilde{Y} \tilde{Z}=Y^2\, , \quad \tilde{Z}=Z 
\end{equation} 
in the equation
\begin{equation}
 SE(\lambda)^{2r}_{r,p,q}:  \quad Y^{2r} \, Z^{r-p+q}= X^{p+q-r} \; (X-Z)^{2r-p} \; (X-\lambda \, Z)^{2r-p} \;.
\end{equation}
We obtain the equation
\begin{equation}
 SE(\lambda)^{r}_{r,p,q}:  \quad \tilde{Y}^{r} \, \tilde{Z}^{2r-p+q}= \tilde{X}^{p+q-r} \; (\tilde{X}-\tilde{Z})^{2r-p} \; (\tilde{X}-\lambda \, \tilde{Z})^{2r-p}  \;.
\end{equation}
Using Lemma~\ref{projective_model} we can construct the minimal resolution $C'_1$ of genus $r-1$ for the singular curve $SE(\lambda)^{r}_{r,p,q}$ denoted by
$\psi'_1: C'_1\to SE(\lambda)^{r}_{r,p,q}$,  and the ramified degree-$r$ map $[\tilde{X}:\tilde{Y}:\tilde{Z}] \mapsto [\tilde{X}:\tilde{Z}]$ denoted by $\phi'_1:SE(\lambda)^{r}_{r,p,q} \to \mathbb{P}^1$
with four totally ramified points. The Puiseux expansions about all singular points of $SE(\lambda)^{r}_{r,p,q}$ defining the local normalization maps are given in Table~\ref{tab:Puiseux_b}.
 \begin{table}[H]
\scalebox{0.55}{
\begin{tabular}{|c|c|c|c|c|}
\hline
&&&&\\[-0.8em]
$[\tilde{X}:\tilde{Y}:\tilde{Z}]$ & $[0:0:1]$ &  $[1:0:1]$ & $[\lambda:0:1]$ & $[0:1:0]$\\[0.2em]
\hline
&&&&\\[-0.8em]
multiplicity & $p+q-r$ & $2r-p$ & $2r-p$ & $r-p+q$\\[0.2em]
\hline
&&&&\\[-0.8em]
$\begin{array}{c} \text{Puiseux} \\ \text{coefficients} \end{array}$
 & $(r,p+q-r)$ 
 & $(r,2r-p)$ 
 & $(r,2r-p)$ 
 & $(3r-p+q,2r-p+q)$ \\[0.2em]
\hline
&&&&\\[-0.8em]
$\begin{array}{c} \text{Puiseux} \\ \text{expansion} \end{array}$
& $\begin{array}{l} \tilde{X}= \tilde{z}_0^{r} \\[0.2em] \tilde{Y} = \lambda^{\frac{2r-p}{r}} \, \tilde{z}_0^{p+q-r} \, \big(1 + \tilde{y}_0(\tilde{z}_0^{r}) \big)\\[0.2em] \tilde{Z}=1 \end{array}$ 
& $\begin{array}{l} \tilde{X}= 1 + \tilde{z}_1^{r} \\[0.2em] \tilde{Y} = (1-\lambda)^{\frac{2r-p}{r}} \, \tilde{z}_1^{2r-p}  \, \big(1 + \tilde{y}_1(\tilde{z}_1^{r}) \big)\\[0.2em] \tilde{Z}=1 \end{array}$ 
& $\begin{array}{l} \tilde{X}= \lambda + \tilde{z}_\lambda^{r} \\[-0.1em] \tilde{Y} = \big(\lambda (\lambda-1)\big)^{\frac{2r-p}{r}} \, \tilde{z}_\lambda^{2r-p}  \, \big(1 + \tilde{y}_\lambda(\tilde{z}_\lambda^{r}) \big) \\[0.4em]
 \tilde{Z}=1\end{array}$
& $\begin{array}{l} \tilde{X}= \tilde{z}_\infty^{2r-p+q}  \, \big(1 + x_\infty(\tilde{z}_\infty^{r}) \big) \\ \tilde{Y} = 1 \\ \tilde{Z}= \tilde{z}_\infty^{3r-p+q} \end{array}$  \\[0.2em]
\hline
\end{tabular}}
\caption{Puiseux expansions around singular points}\label{tab:Puiseux_b}
\end{table}
We observe that the Puiseux expansion in Table~\ref{tab:Puiseux} for $SE(\lambda)^{2r}_{r,p,q}$ coincide with the Puiseux expansion in Table~\ref{tab:Puiseux_b} for $SE(\lambda)^{r}_{r,p,q}$
when using Equations~(\ref{transfo_quotient}) together with the transformation between the local normalization parameters given by
\begin{equation}
\label{new_normalizaton}
 \tilde{z}_0 = z_0^2 , \quad \tilde{z}_1 = z_1^2 , \quad \tilde{z}_\lambda = z_\lambda^2, \quad  \tilde{z}_\infty = z_\infty^2 \;.
\end{equation}
Notice that by construction from the defining equations for $SE(\lambda)^{2r}_{r,p,q}$ and $SE(\lambda)^{r}_{r,p,q}$ we have $(1+y_0(t))^2=1+\tilde{y}_0(t)$, $(1+y_1(t))^2=1+\tilde{y}_1(t)$, 
and $(1+y_\lambda(t))^2=1+\tilde{y}_\lambda(t)$. In Lemma~\ref{Lem:lift} we proved that the involution automorphism~(\ref{involution_EC}) of the action $G_0=\mathbb{Z}_2$ defined on 
$SE(\Lambda_1^2)^{2r}_{r,p,q}$ lifts to the minimal resolution as $z_0, z_1,z_\lambda, z_\infty \mapsto - z_0, -z_1,-z_\lambda, -z_\infty$. As Equation~(\ref{new_normalizaton}) is invariant
under this action, the minimal resolution of $SE(\lambda)^{r}_{r,p,q}$ is the quotient curve $C_1/G_0$.
\end{proof}

\begin{proof}
The generalized Kummer surface $S_0$ is the minimal resolution of $V_0 = T/G_0$ with $T=C_1 \times C_2$ and $G_0=\mathbb{Z}_2$ obtained 
by blowing up $16$ ordinary double point singularities (cf.~Lemma~\ref{Lem:RDPsingularities}). $V_0$ is the singular double cover of 
$C_1/G_0 \times C_2/G_0$ whose branch curve is a union of vertical and horizontal curves.
By blowing up the branch locus we obtain $S_0$ still having a map to $C_1/G_0$.
This establishes the isotrivial fibration $f'_1: S_0 \to C_1/G_0$.
The other isotrivial fibration  is obtained in an analogous manner.
The composition 
$$
C_1/G_0 = C'_1  \overset{\psi_1'\;}{\dashrightarrow} SE(\Lambda_1^2)^{r}_{r,p,q} \overset{\phi_1'}{\longrightarrow} \mathbb{P}^1 
$$
maps every point on the base curve $C'_1$ of $f'_1$ to a point in $\mathbb{P}^1$ with affine coordinate $\tilde{U}$.  For each fiber
away from the ramification points, i.e., $\tilde{U} \not \in \lbrace 0, 1, \Lambda_1^2, \infty\rbrace$,  the fiber coordinates when composed with the normalization 
map $\psi_2: C_2 \to SE(\Lambda_2^2)^{2r}_{r,2r-p,2r-q}$ then satisfy Equation~(\ref{IsoFibration_2}).
\end{proof}
Similarly, we have the following lemma describing the non-isotrivial fibration on a generalized Kummer variety:
\begin{proposition}
\label{Lem:SW}
Assume that $r>2$. On the generalized Kummer variety $S_0$ 
there is a fibration $\tilde{f}: S_0 \to \tilde{C}$ such
that the genus of the fibres equals $2r-1$ and $\tilde{C}$ is a smooth irreducible curve of genus $r-1$. The base curve $\tilde{C}$ admits
a meromorphic map to $\mathbb{P}^1$ such that the general fiber is
the superelliptic curve associated with the singular curve
\begin{equation}
\label{fibrationJ7fib2}
\begin{split}
 y^{2r} = (u-A)^{2r-p} \, (u-B)^{2r-p} \, x^{3r-p-q} \, \left(x^2+ 2 \, (1-2\, u) \, x + 1\right)^p \;,
\end{split}
\end{equation}
where the affine coordinate on the base curve is $u \in \mathbb{P}^1$,
and the moduli are related by
\begin{equation}
\label{RelationModuli}
 A = \frac{(\Lambda_1 + \Lambda_2)^2}{4 \, \Lambda_1 \, \Lambda_2} \;, \qquad B = \frac{(\Lambda_1 \, \Lambda_2 +1)^2}{4 \, \Lambda_1 \, \Lambda_2} \;.
\end{equation}
It follows that
\begin{equation}
\label{Relation_2forms}
  \pi^* \left(du \wedge \frac{dx}{y}\right) =   \frac{2^{2-\frac{2p}{r}}  \, \Lambda_1^{\frac{3}{2}-\frac{p}{2r}-\frac{q}{2r}} \,  \Lambda_2^{\frac{1}{2}-\frac{p}{2r}+\frac{q}{2r}}}{(\Lambda_1^2-1)^{1-\frac{p}{r}}}
 \;   \frac{d\zeta_1}{\eta_1} \boxtimes \frac{d\zeta_2}{\eta_2}  \;.
\end{equation}
In particular, $du \wedge dx/y$ is a holomorphic differential two-form on $S_0$.
\end{proposition}
\begin{proof}
The automorphism of $V_0$ given by Equations~(\ref{fibrationJ7param}) and~(\ref{fibrationJ7coords}) lifts to a diffeomorphism $\Phi$ of $S_0$.
Since the isotrivial fibration in Proposition~\ref{IsoFibration} is a genus $r-1$ pencil with $r-1\ge 2$, then the Isotropic Subspace Theorem~\cite[Thm.~1.10]{MR1098610} 
implies that $\Phi \, f'_1$ determines a genus $r-1$ pencil as well. Then it follows from Seiberg-Witten theory by ~\cite[Thm.~2.9]{MR1737256} 
that the genus of the fibres of $f'_1$ and $\tilde{f}$ are equal. The necessary condition of $\phi$ being orientation preserving follows from Equation~(\ref{relationJ72form}).

The projection map in Equation~(\ref{fibrationJ7param}) from a point on $T=C_1 \times C_2$ to $U\in \mathbb{P}^1$ 
 -- where $\zeta_1$ and $\zeta_2$ are understood to be composed with their respective normalization maps $\psi_1: C_1\to SE(\Lambda_1^2)^{2r}_{r,p,q}$ 
and $\psi_2: C_2\to SE(\Lambda_2^2)^{2r}_{r,2r-p,2r-q}$ for their minimal resolutions -- descends to a surjective map $S_0 \to \mathbb{P}^1$ such that
$X, Y$ and $U$ defined in Equation~(\ref{fibrationJ7coords}) and Equation~(\ref{fibrationJ7param}), respectively,  are rational functions on the fiber and base curve, respectively.
The transformation
\begin{equation}
\label{transfo_x_X}
\begin{split}
x = \frac{X}{\Lambda_1 \, \Lambda_2 \, U \, (U-1)^2} \;, \quad y =- \frac{(A-B)^{2-\frac{p}{r}} \, \Lambda_2^{1-\frac{q}{r}}\, Y}{(\Lambda_1 \, \Lambda_2)^{\frac{3}{2}+\frac{p}{2r}-\frac{q}{2r}} \, (1-\Lambda_2^2)^{1-\frac{p}{r}} \, U \, (U-1)^4}\;,
\end{split}
\end{equation}
is an isomorphism on each fiber for $u \not \in \lbrace 0, 1, A, B,  \infty\rbrace$. Together with Equations~(\ref{RelationModuli}) and 
\begin{equation}
\label{transfo_u_U}
u = B + \frac{B-A}{U-1} \;,
\end{equation}
they relate Equation~(\ref{fibrationJ7fib2}) to Equation~(\ref{fibrationJ7fib}). Therefore, using Lemma~\ref{SWmodel} any fiber over $u \not \in \lbrace 0, 1, A, B,  \infty\rbrace$ 
is a (rescaled) superelliptic curve $SE(\Lambda^2)^{2r}_{r,2r-p,2r-q}$ with $u=(1+\Lambda)^2/(4 \, \Lambda)$ whose minimal resolution is a smooth irreducible curve of genus $2r-1$.
Equation~(\ref{Relation_2forms}) then follows from Equation~(\ref{relationJ72form}) and the transformations in Equations~(\ref{transfo_x_X})
and~(\ref{transfo_u_U}). Moreover, Equation~(\ref{Relation_2forms}) proves that up to a non-vanishing scalar that only depends on the moduli the two-form $du \wedge dx/y$ equals
the holomorphic two-form $\omega$ already constructed in Section~\ref{SSec:SuperellipticKummer}.
\end{proof}
\begin{remark}
For $r=p=q=1$,  Equation~(\ref{fibrationJ7fib2}) defines the Jacobian elliptic fibration $\mathfrak{J}_7$ on a Kummer
surface of two non-isogenous elliptic curves with singular fibers of Kodaira-type $I_0^*$ over $u=A, B$, and of type $I_1$ and $I_4^*$ over $u=0, 1$ and $u=\infty$. 
It is the quadratic twist of the extremal Jacobian elliptic rational surface that is the modular elliptic surface for $\Gamma_0(4)$ (cf.~Remark~\ref{ModularSurfaces}).
\end{remark}

We can also start with the \emph{twisted Legendre pencil of superelliptic curves} given by
\begin{equation}
\label{legK3_mu}
  y^{2r} = (u-A)^{2r-q} \, (u-B)^{2r-q} \; x^{2r-p} \, (x-1)^{2r-p} \, (x-u)^{p+q-r} \;.
\end{equation}
We want to determine under what circumstances the generalized fibration $\mathfrak{J}_6$
from Section~\ref{TwistFibJ6} is the pull-back of the twisted Legendre pencil of superelliptic curves 
via a morphism $\mathbb{P}^1 \to \mathbb{P}^1$, i.e., a fiber product $S_0 \times_{\mathbb{P}^1}
\mathbb{P}^1$. If we apply the transformation
\begin{equation}
\begin{split}
 u & = \frac{1}{1-\tilde{u}} \;,\\
 x & = \frac{1}{1-\tilde{x}} \;,\\ 
 y & = \frac{\tilde{y}}{(1-\tilde{x})^2\, (1-\tilde{u})^2 \, \big((1-\tilde{a})(1-\tilde{b})\big)^{1-\frac{q}{2r}}} \;,\\
 (A,B) & = \left( \frac{1}{1-\tilde{A}}, \frac{1}{1-\tilde{B}} \right) \;,
\end{split}
\end{equation}
we obtain the family
\begin{equation}
  \tilde{y}^{2r} = (\tilde{u}-1)^{-p+q+r} \, (\tilde{u}-\tilde{A})^{2r-q} \, (\tilde{u}-\tilde{B})^{2r-q} \; \tilde{x}^{2r-p} \, (\tilde{x}-1)^{p-q+r} \, (\tilde{x}-\tilde{u})^{p+q-r} \;.
\end{equation}
We now assume that $\tilde{u}$ is given by a base change via the degree-two map
\begin{equation}
\label{degree-two-map}
 U \in \mathbb{P}^1 \mapsto \tilde{u} = \frac{(U-1)(U-\alpha^2 \beta^2)}{(U-\alpha^2)(U-\beta^2)} \in \mathbb{P}^1 \;,
\end{equation}
such that the pre-images of $\tilde{u}=0, 1, \infty$ are given by $(1, \alpha^2 \beta^2)$, $(0,\infty)$, and $(\alpha^2, \beta^2)$, respectively.
The branching points, i.e., points where $\partial_{U} \tilde{u}=0$, are $U=\pm \alpha \beta$, and
the corresponding ramification points are $\tilde{B}=(\alpha\beta-1)^2/(\alpha-\beta)^2$ and $\tilde{A}=(\alpha\beta+1)^2/(\alpha+\beta)^2$, respectively.
Both have ramification index $2$. Hence, by the Riemann-Hurwitz formula the map in Equation~(\ref{degree-two-map}) is indeed a morphism from $\mathbb{P}^1 \to \mathbb{P}^1$.
Therefore, if we set
\begin{equation}
\begin{split}
 \tilde{u} & =  \frac{(U-1)(U-\alpha^2 \beta^2)}{(U-\alpha^2)(U-\beta^2)} \;, \\
 \tilde{x} & =  \frac{\beta^2 \, X}{U \, (U-\alpha^2) \, (U-\beta^2)} \;,\\ 
 \tilde{y} & =  \frac{ \beta^{3+\frac{p}{r}-\frac{q}{r}} \, \big( (\alpha^2-1) \, (1-\beta^2)\big)^{\frac{5}{2}-\frac{p}{2r}-\frac{q}{2r}}\, (U^2-\alpha^2\beta^2) \, Y }{U \, (U-\alpha^2)^3 \, (U-\beta^2)^3 \, (\alpha^2-\beta^2)^{2-\frac{q}{r}}} \;,\\
 (\tilde{A},\tilde{B}) & = \left( \frac{(\alpha\beta+1)^2}{(\alpha+\beta)^2},  \frac{(\alpha\beta-1)^2}{(\alpha-\beta)^2} \right) \;,
\end{split}
\end{equation}
and  let $(\alpha,\beta) = (\Lambda_1, 1/\Lambda_2)$, we obtain
\begin{equation}
\begin{split}
 Y^{2r} = \left(1 - \Lambda_1^2\right)^{2(r-p)} \, U^{-2p+q+r} \,  \left(U-\Lambda_1^2\right)^{q-r} \,  \left(\Lambda_2^2 \, U - 1\right)^{q-r} \,\left(\Lambda_2^2 \, U^2 - \Lambda_1^2\right)^{2(r-q)}  \, X^{2r-p} \\
 \times \, \left( X - U \, (U-\Lambda_1^2) \, (\Lambda_2^2 \, U -1) \right)^{p+q-r} \,  \left( X - U \, (U-1) \, (\Lambda_2^2 \, U -\Lambda_1^2) \right)^{p+q-r} \;.
\end{split}
\end{equation}
Thus, we have proved the following lemma:
\begin{lemma}
\label{Lem:6lines}
Equation~(\ref{fibrationJ6fib}) on $S_0^{(r,p,q=r)}(\Lambda_1^2,\Lambda_2^2)$ is the pull-back of  the twisted Legendre pencil of superelliptic curves given by
\begin{equation}
\label{fibrationJ6fib2}
\begin{split}
  y^{2r} = (u-A)^{r} \, (u-B)^{r} \; x^{2r-p} \, (x-1)^{2r-p} \, (x-u)^{p} \;,
\end{split}
\end{equation}
via the degree-two base change $g: \mathbb{P}^1 \to \mathbb{P}^1, U \mapsto u$ given by
\begin{equation}
\label{base_change}
 \frac{1}{1-u} =  \frac{(U-1)(U-\alpha^2 \beta^2)}{(U-\alpha^2)(U-\beta^2)} 
\end{equation}
with $(\alpha,\beta) = (\Lambda_1, 1/\Lambda_2)$.
The moduli of the Legendre pencil are related to the moduli of $S_0^{(r,p,q=r)}(\Lambda_1^2,\Lambda_2^2)$ by the rational maps
\begin{equation}
\label{RelationModuli2}
 A = \frac{(\Lambda_1 \, \Lambda_2 +1)^2}{\left(\Lambda_1^2-1\right) \, \left(\Lambda_2^2-1\right)} \;, \quad
 B = \frac{(\Lambda_1 \, \Lambda_2 -1)^2}{\left(\Lambda_1^2-1\right) \, \left(\Lambda_2^2-1\right)} \;.
\end{equation}
\end{lemma}
\hfill$\square$

\begin{remark}
The family 
\begin{equation}
\label{twist2res_genb}
  X_3^{2r} =X_1^{2r-p} \, (1-X_1)^{2r-p} \,  X_2^{2r-q} \, (1-X_2)^{2r-q} \;  (1- z_1 \, X_1 - z_2 \, X_2)^{p+q-r} 
\end{equation}
becomes isomorphic to family~(\ref{legK3_mu}), i.e., 
\begin{equation}
\label{leg$K3$_mu_a}
  y^{2r} = (u-A)^{2r-q} \, (u-B)^{2r-q} \; x^{2r-p} \, (x-1)^{2r-p} \, (x-u)^{p+q-r} \;
\end{equation}
by setting $X_1=x$ and
\begin{equation}
\label{transformb}
 \begin{split}
 (z_1, z_2) & =  \big( \frac{1}{A}, 1 - \frac{B}{A} \big) \;,\\
  X_2 &=\frac{u-A}{B-A},\\
  i \, y& = X_3 \, A^{\frac{1}{2}-\frac{p}{2r}-\frac{q}{2r}} \, \, (A-B)^{2-\frac{q}{r}} \;.
   \end{split}
\end{equation}
Lemma~(\ref{Lem:6lines}) then shows that Equation~(\ref{twist2res_genb}) coincides with Equation~(\ref{fibrationJ6fib}) describing the generalized fibration $\mathfrak{J}_6$ if and only if $q=r$.
On the other hand, interchanging the roles $X_1 \leftrightarrow X_2$ and $z_1 \leftrightarrow z_2$, we obtain the family
\begin{equation}
\label{leg$K3$_mu_b}
  y^{2r} = (u-A)^{2r-p} \, (u-B)^{2r-p} \; x^{2r-q} \, (x-1)^{2r-q} \, (x-u)^{p+q-r} \;.
\end{equation}
The transformation
\begin{equation}
\label{transfo}
\begin{split}
 x & \mapsto \frac{x^2 + 2 \, (1-2 \, u) \, x +1}{4\, x}  + u \;, \\
 y & \mapsto \frac{(x^2+1)^{2-\frac{q}{r}} \, \left(x^2 + 2\, (1-2 \, u) \, x +1\right)^{\frac{q}{2r}-\frac{1}{2}} \, y}{2^{3+\frac{p}{r}-\frac{q}{r}} \, x^{3-\frac{q}{r}}}
\end{split}
\end{equation}
maps family~(\ref{leg$K3$_mu_b}) to Equation~(\ref{fibrationJ7fib2}) describing the generalized fibration $\mathfrak{J}_7$.
The transformation~(\ref{transfo}) is rational if and only if $q=r$. Note that in the process of
relating Equations~(\ref{fibrationJ6fib}) and (\ref{fibrationJ7fib2}) we have thus interchanged the roles of affine base and fiber coordinates
and simultaneously the roles of the superelliptic curves $C_1$ and $C_2$, and applied fiberwise the rational transformation
from Lemma~\ref{SWmodel}.
\end{remark}

\begin{remark}
For $r=p=q=1$,  Equation~(\ref{fibrationJ6fib2}) defines a Jacobian elliptic fibration, called $\mathfrak{J}_6$ in \cite{MR2409557}, on a Kummer
surface of two non-isogenous elliptic curves with $2$ singular fibers of Kodaira-type $I_2^*$ and $4 $ singular fibers of type $I_2$. 
This Jacobian elliptic $K3$ surface is obtained by the base transformation in Equation~(\ref{base_change})
from the modular elliptic surface for $\Gamma(2)$ (cf.~Remark~\ref{ModularSurfaces}).
In fact, if we denote by $g: \mathbb{P}^1 \to \mathbb{P}^1$ the degree-two map in Equation~(\ref{base_change}), then the pull-back of the Jacobian rational elliptic surface $S$ 
that is the modular elliptic surface for $\Gamma(2)$ is the Jacobian elliptic fibration $\mathfrak{J}_6$.
On $X$, we  have the deck transformation $\ell$ interchanging the pre-images and the 
fiberwise elliptic involution $\imath: (u,x,y) \mapsto (u,x,-y)$. The composition $\jmath=\imath \circ \ell$ is
a Nikulin involution leaving $du \wedge dx/y$ invariant, and the minimal resolution of the 
quotient $X/\jmath$ is the Jacobian elliptic $K3$ surface $X' \to \mathbb{P}^1$
in Equation~(\ref{legK3_mu}) with $p=q=r=1$.  The $K3$ surface $X'$ is the quadratic twist of the modular elliptic surface $S$ for $\Gamma(2)$ and
thus has fibers of Kodaira-type $I_0^*$ over
the two ramification points of $f$. The situation is summarized in Figure~\ref{fig_geom}. Comparing elliptic surfaces obtained by either a quadratic twist
or a degree-two base transformation was also the basis for the iterative procedure in \cite{Doran:2015aa} that produced families of elliptically fibered Calabi-Yau $n$-folds 
with section from families of elliptic Calabi-Yau varieties of one dimension lower. It also played a key role in relating Yukawa couplings and Seiberg-Witten
 prepotential in \cite{MR2854198}.

\begin{figure}[ht!]
\begin{tikzpicture}[thick,scale=0.7, every node/.style={scale=1.1}]
  \matrix (m) [matrix of math nodes, row sep=3em,column sep=3em]{
  			&  X				&	\\
 S			& \mathbb{P}^1 	&  X'\\
 \mathbb{P}^1	& 				& \mathbb{P}^1	\\};
  \path[-stealth]
  (m-1-2) edge (m-2-1)
  (m-1-2) edge (m-2-3)
  (m-1-2) edge (m-2-2)
  (m-2-1) edge (m-3-1)
  (m-2-3) edge (m-3-3)
  (m-2-2) edge node [above]	{$g$} (m-3-1)
  (m-2-2) edge node [above]	{\phantom{I}$g$} (m-3-3)  
  (m-3-1) edge [-,double distance = 1.5pt](m-3-3);
\end{tikzpicture}
\caption{Relation between twist and elliptic $K3$ surface $\mathfrak{J}_6$}
\label{fig_geom}
\end{figure}
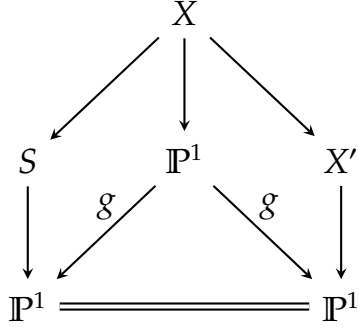
\end{remark}

\section{Period computations on a generalized Kummer surface}
\label{Sec:KummerPeriod}
Equation~(\ref{fibrationJ7fib2}) in Lemma~\ref{Lem:SW} defines a non-isotrivial 
fibration on the generalized Kummer surface $S_0^{(r,2r-p,2r-q)}(\Lambda_1^2,\Lambda_2^2)$ given by
\begin{equation}
\begin{split}
 y^{2r} = (u-A)^{2r-p} \, (u-B)^{2r-p} \, x^{3r-p-q} \, \left(x^2+ 2 \, (1-2\, u) \, x + 1\right)^p 
\end{split}
\end{equation}
where $u$ is the affine coordinate on a base $\mathbb{P}^1$.
We now construct  a cover by pulling-back Equation~(\ref{fibrationJ7fib2}) via the degree-two base change
$$ 
 \tilde{g}: \mathbb{P}^1 \to \mathbb{P}^1, \qquad z \mapsto u= \frac{(1+z)^2}{4z} \;.
$$
We have the following lemma:
\begin{lemma}
\label{Lem:double_cover}
The pencil of superelliptic curves over $\mathbb{P}^1$ given by 
\begin{equation}
\label{fibrationJ7fib3}
\begin{split}
\tilde{Y}^{2r} = \; & z^{p+q-r} \, \left( \big(z - a\big) \, \big(z-\frac{1}{a}\big) \, \big(z - b\big) \, \big(z-\frac{1}{b}\big)\right)^{2r-p} \\
& \times \; 2^{4p} \; \tilde{X}^{3r-p-q} \, \big(\tilde{X}-1\big)^p \, \big(\tilde{X}-z^2\big)^p \;,
\end{split}
\end{equation}
is the pull-back of Equation~(\ref{fibrationJ7fib2}) on $S_0^{(r,2r-p,2r-q)}(\Lambda_1^2,\Lambda_2^2)$
via the degree-two base change $\tilde{g}: \mathbb{P}^1 \to \mathbb{P}^1, z \mapsto u$ with the following relations between the affine coordinates on the base and moduli:
\begin{equation}
\label{moduli_DoubleCover}
 u= \frac{(1+z)^2}{4z} , \quad A = \frac{(1+a)^2}{4 \, a}, \quad  B = \frac{(1+b)^2}{4 \, b}.
\end{equation}
We also have
\begin{equation}
  \tilde{g}^* \left(du \wedge \frac{dx}{y}\right) =   4 \, (z^2-1) \, dz\wedge \frac{d\tilde{X}}{\tilde{Y}} \;.
\end{equation}
\end{lemma}
\begin{proof}
By direct computation using $x= \tilde{X}/z$, $y=\tilde{Y}/(16 \, z^3)$.
\end{proof}
\begin{remark}
The comparison of Equations~(\ref{moduli_DoubleCover}) with Equations~(\ref{RelationModuli}) shows that
$a=\Lambda_2/\Lambda_1$ or $a=\Lambda_1/\Lambda_2$ and $b=\Lambda_1 \,\Lambda_2$ or  $b=1/(\Lambda_1 \,\Lambda_2)$.
\end{remark}
The family of superelliptic curves over $\mathbb{P}^1$ defined by Equation~(\ref{fibrationJ7fib3}) enables us to define a family of closed two-cycles $\Sigma^{(k,l)}_z$ 
with $1 \le k,l \le 2r-1$ on its total space as follows: first, define a closed one-cycle in each fiber over $z \not \in \lbrace 0, 1, a^{\pm1}, b^{\pm1}, \infty \rbrace$ of the 
total space as a clockwise one-cycle
around the line segment $[0,z^2]$ in the $\tilde{X}$-plane on a chosen $\tilde{Y}$-sheet. However, as we move around in the $z$-plane (which represents 
moving along the base curve) the overall complex factor in Equation~(\ref{fibrationJ7fib3}) changes.  Thus, for $\tilde{Y}$ to remain a well-defined
multi-valued function we have to introduce branch cuts in the $z$-plane as well. We define a multi-valued function $\tilde{Y}$ by writing
\begin{equation}
\label{branch}
 \tilde{Y} = 2^{4\beta_2} \; z^{\beta_1+\beta_2-\frac{1}{2}} \,  \prod_{z_0 \in \lbrace a^{\pm 1}, b ^{\pm 1}\rbrace} \big(z - z_0\big)^{1-\beta_2} \;
 \tilde{X}^{\frac{3}{2}-\beta_1-\beta_2} \, \big(1-\tilde{X}\big)^{\beta_2} \, \big(z^2 - \tilde{X}\big)^{\beta_2} 
\end{equation}
with  $\beta_1=q/(2r)$, $\beta_2=p/(2r)$. We choose one branch cut in the $z$-plane connecting the points $z=a$ and $z=b$ while not intersecting any other
branch points and cuts. We then define the closed two-cycle $\Sigma^{(k,l)}_z$ given two integers $1 \le k,l \le 2r-1$ on the total space 
as the closed one-cycle that runs suitably close in clockwise orientation around the line segment $[a,b]$
on the $k^{\text{th}}$ branch of the $z$-plane for the multi-valued function
$$
 z^{\beta_1+\beta_2-\frac{1}{2}} \,  \prod_{z_0 \in \lbrace a^{\pm 1}, b ^{\pm 1}\rbrace} \big(z - z_0\big)^{1-\beta_2}, 
$$ 
while taking a clockwise one-cycle around the line segment $[0,z^2]$ on the $l^{\text{th}}$ branch of the $\tilde{X}$-plane for the multi-valued function
$$
  \tilde{X}^{\frac{3}{2}-\beta_1-\beta_2} \, \big(1-\tilde{X}\big)^{\beta_2} \, \big(z^2-\tilde{X}\big)^{\beta_2} \;.
$$
We then set $\Sigma^{(k,l)}_u = \tilde{g}_* \Sigma^{(k,l)}_z$ for $1 \le k,l \le 2r-1$.
We have the following lemma:
\begin{lemma}
\label{Lem:F2period}
The period of the holomorphic two-form $du\wedge dx/y$ over the two-cycle $\Sigma^{(k,l)}_u$ on $S_0^{(r,2r-p,2r-q)}(\Lambda_1^2,\Lambda_2^2)$ equals
\begin{equation}
\begin{split}
\oiint_{\Sigma^{(k,l)}_u}  du\wedge \frac{dx}{y}
 = \; (-1)^{\beta_2} \, C^{(k)}_{2r}  \, C^{(l)}_{2r}  \, \frac{\Gamma\left(\alpha\right) \, \Gamma\left(1-\beta_2\right) \, \Gamma(\beta_2)^2}{4^{\alpha}  \, \Gamma\left(\beta_1 + \frac{1}{2}\right) \, \Gamma(2\, \beta_2)}\\ \;
 \frac{1}{A^{\alpha} \, (A-B)^{1-2 \, \beta_2}} \;  \app2{\alpha;\;\beta_1,\beta_2}{2 \beta_1, \, 2 \beta_2}{\frac{1}{A}, \, 1 - \frac{B}{A}} \qquad \qquad
 \end{split}
\end{equation}
where $C^{(k)}_{2r}= (\rho_{2r}-1)/\rho_{2r}^k$, $(-1)^{\beta_2}=\rho_{4r}^{p}$, $\rho_{2r}=\exp{(\frac{2\pi i}{2r})}$, and $1 \le k,l \le 2r-1$, and the relation between $A, B$ and $\Lambda_1, \Lambda_2$ is given in
Equations~(\ref{RelationModuli}).
\end{lemma}
\begin{proof}
We start by transforming the double integral into an iterated integral with each contour reduced to an integration along a branch cut (analogous to 
the approach in the proof of Lemma~\ref{SuperellipticPeriods}). We obtain
\begin{equation}
 \begin{split}
& \frac{1}{C^{(k)}_{2r}  \,  C^{(l)}_{2r}} \,  \oiint_{\Sigma^{(k,l)}_z} \; 4 \, (z^2-1) \; dz\wedge \frac{d\tilde{X}}{\tilde{Y}}   =  \frac{1}{C^{(k)}_{2r}  \,  C^{(l)}_{2r}} \, \oiint_{\Sigma^{(k,l)}_u}  du\wedge \frac{dx}{y} \\ 
= & \;  \int_A^B \frac{du}{(u-A)^{1-\beta_2} \, (u-B)^{1-\beta_2}} \, \int_0^z \frac{dx}{x^{\frac{3}{2}-\beta_1-\beta_2} \, \left(x^2+ 2 \, (1-2\, u) \, x + 1\right)^{\beta_2} } \\
= & \;  -  \, e^{\pi i\beta_2} \,\int_A^B \frac{du}{(A-u)^{1-\beta_2} \, (u-B)^{1-\beta_2}} \, \int_0^z \frac{dx}{x^{\frac{3}{2}-\beta_1-\beta_2} \, \left(x^2+ 2 \, (1-2\, u) \, x + 1\right)^{\beta_2} } \;.
\end{split}
\end{equation}
Using the relations in Equation~(\ref{moduli_DoubleCover}) and the branch of $y$ compatible with Equation~(\ref{branch}), the inner integral evaluates to
\begin{equation}
\label{nr_eq2}
 \begin{split}
& \; \int_0^z \frac{dx}{x^{\frac{3}{2}-\beta_1-\beta_2} \, \left(x^2+ 2 \, (1-2\, u) \, x + 1\right)^{\beta_2} }  =   \int_0^z \frac{dx}{x^{\frac{3}{2}-\beta_1-\beta_2} \, \left(z - x\right)^{\beta_2} \, \left( \frac{1}{z} - x\right)^{\beta_2} }\\
& \; =  \,  z^{\beta_1+\beta_2-\frac{1}{2}} \,  \int_0^1 \frac{d\tilde{x}}{\tilde{x}^{\frac{3}{2}-\beta_1-\beta_2} \, \left(1 - \tilde{x}\right)^{\beta_2} \,  \left(1 - z^2 \,\tilde{x}\right)^{\beta_2} } \\
& \; =  \,  z^{\beta_1+\beta_2-\frac{1}{2}} \,  \frac{\Gamma\left(\beta_1 + \beta_2 - \frac{1}{2}\right) \, \Gamma\left(1-\beta_2\right)}{\Gamma\left(\beta_1 + \frac{1}{2}\right)} \;   \hpg21{\beta_1 + \beta_2 - \frac{1}{2},\,\beta_2}{\beta_1 + \frac{1}{2}}{z^2} \;.
\end{split}
\end{equation}  
Note that the quadratic identity~\cite[Eq.~(15.3.17)]{MR0167642} is equivalent to the following identity for the Gauss' hypergeometric function
\begin{equation}
  (1+z)^{2\beta_1+2\beta_2-1} \;  \hpg21{\beta_1 + \beta_2 - \frac{1}{2},\,\beta_2}{\beta_1 + \frac{1}{2}}{z^2} =  \hpg21{\beta_1 + \beta_2 - \frac{1}{2},\,\beta_1}{2 \, \beta_1}{\frac{1}{u}} \;
\end{equation}
with $z$ and $u$ related by Equation~(\ref{moduli_DoubleCover}).
Applying this identity to Equation~(\ref{nr_eq2}) we obtain
\begin{equation}
\label{nr_eq}
 \begin{split}
& \qquad \qquad  \int_0^z \frac{dx}{x^{\frac{3}{2}-\beta_1-\beta_2} \, \left(x^2+ 2 \, (1-2\, u) \, x + 1\right)^{\beta_2} } \\
 \; =   \, & \,  \frac{\Gamma\left(\beta_1 + \beta_2 - \frac{1}{2}\right) \, \Gamma\left(1-\beta_2\right)}{\Gamma\left(\beta_1 + \frac{1}{2}\right)} \;  \left(\frac{1}{4\, u}\right)^{\beta_1+\beta_2-\frac{1}{2}} \,  \hpg21{\beta_1 + \beta_2 - \frac{1}{2},\,\beta_1}{2 \, \beta_1}{\frac{1}{u}} \;.
\end{split}
\end{equation}  
Combining this result with Equation~(\ref{IntegralTransform}) of Corollary~\ref{EulerIntegralTransform}, we obtain
\begin{equation}
 \begin{split}
&   \oiint_{\Sigma^{(k,l)}_z} \; 4 \, (z^2-1) \; dz\wedge \frac{d\tilde{X}}{\tilde{Y}}   =  \oiint_{\Sigma^{(k,l)}_u}  du\wedge \frac{dx}{y} 
=  \;   - \, C^{(k)}_{2r} \, C^{(l)}_{2r}  \,   \frac{\Gamma\left(\alpha\right) \, \Gamma\left(1-\beta_2\right)}{4^{\alpha}  \, \Gamma\left(\beta_1 + \frac{1}{2}\right)} \\
& \qquad \times  \,   e^{\pi i\beta_2} \, \int_A^B \frac{du}{(A-u)^{1-\beta_2} \, (u-B)^{1-\beta_2} \, u^{\alpha}} \; \hpg21{\beta_1 + \beta_2 - \frac{1}{2},\,\beta_1}{2 \, \beta_1}{\frac{1}{u}} \\[0.4em]
= \, &  e^{\pi i\beta_2} \, C^{(k)}_{2r}  \,  \frac{\Gamma\left(\alpha\right) \, \Gamma\left(1-\beta_2\right) \, \Gamma(\beta_2)^2}{4^{\alpha}  \, \Gamma\left(\beta_1 + \frac{1}{2}\right) \, \Gamma(2\, \beta_2)} \;
 \frac{1}{A^{\alpha} \, (A-B)^{1-2 \, \beta_2}} \;
 \app2{\alpha;\;\beta_1,\beta_2}{2 \beta_1, \, 2 \beta_2}{\frac{1}{A}, \, 1 - \frac{B}{A}}  \;.
\end{split}
\end{equation}
\end{proof}
We now relate the periods over the two-cycles $\Sigma^{(k,l)}_u$ -- defined naturally from the point of view of the generalized fibration $\mathfrak{J}_7$ 
 -- to the periods over the two-cycles $\sigma_{i,j}$ -- defined naturally from the point of view of the generalized Kummer construction 
in Section~\ref{SSec:SuperellipticKummer}.  
We have the following lemma:
\begin{lemma}
\label{Lem:2cycles}
On the generalized Kummer surface $S_0^{(r,2r-p,2r-q)}(\Lambda_1^2,\Lambda_2^2)$, the periods of the holomorphic two-form $\omega$ over the two-cycles 
$\sigma_{i,j}\phantom{^{(P)}}$\hspace*{-0.4cm} (defined in Section~\ref{SSec:SuperellipticKummer}) are related to the periods of the holomorphic two-form $dU \wedge dX/Y$ over two-cycles 
$\Sigma^{(i+k-1,j+r+1-k)}_u$ by
\begin{equation}
\label{period_eqn}
 \oiint_{ \sigma_{i,j}} \omega =  \oiint_{ \Sigma^{(i+k-1,j+r+1-k)}_u} dU \wedge \frac{dX}{Y} 
\end{equation}
with $1 \le i, j \le r-1$ and for all $k$ with $1 \le k \le r-1$.
\end{lemma}
\begin{proof}
We will use the variables $U$ and $X$ from Equations~(\ref{fibrationJ7param}) and~(\ref{fibrationJ7coords}), i.e.,
\begin{equation}
\label{nr1}
\begin{split}
 U & = \frac{(\zeta_2 - \Lambda_2^2) \, (\zeta_1 - \zeta_2)}{(\zeta_2 -1) \, (\Lambda^2_2 \, \zeta_1 - \zeta_2)} \;,\\
 x &= \frac{X}{\Lambda_1 \, \Lambda_2 \, U \, (U-1)^2} = \frac{\Lambda_2 \, \zeta_1}{\Lambda_1 \, \zeta_2} \;.
\end{split}
\end{equation}
Note that because of Equation~(\ref{transfo_u_U}) the limits $u=A$ and $u=B$ are equivalent to $U=0$ and $U=\infty$, respectively.
Solving Equations~(\ref{nr1}) for $(\zeta_1, \zeta_2)$, we obtain
\begin{equation}
\begin{split}
 \zeta_1 &= \frac{\Lambda_1 \, x \, \left( \Lambda_1 \, \Lambda_2 \, U \, x - \Lambda_1 \, \Lambda_2 \, x - U +\Lambda_2^2\right)}{ \Lambda_1 \, \Lambda_2^2 \, U \, x  - \Lambda_1 \, x -\Lambda_2 \, U +\Lambda_2} \;,\\
 \zeta_2 & = \frac{\Lambda_2 \left( \Lambda_1 \, \Lambda_2 \, U \, x - \Lambda_1 \, \Lambda_2 \, x - U +\Lambda_2^2\right) }{ \Lambda_1 \, \Lambda_2^2 \, U \, x  - \Lambda_1 \, x -\Lambda_2 \, U +\Lambda_2} \;.
 \end{split}
\end{equation}
Thus, the line with $U=0$ and $x$ unrestricted corresponds to the line with $\zeta_2=\Lambda_2^2$ and $\zeta_1=\Lambda_1^2 \, \Lambda_2^2 \, x$, i.e.,  $\zeta_1$ unrestricted. 
Similarly, the line with $U=\infty$ and $x$ unrestricted  corresponds to the line with $\zeta_2=1$ and $\zeta_1=\Lambda_1^2 \, x$, i.e., $\zeta_1$ unrestricted. On the other hand, 
the line with $x=0$ and $U$ unrestricted corresponds to the line with $\zeta_1=0$ and $\zeta_2$ unrestricted.
Similarly, the line with $x=z$ and $U$ unrestricted corresponds to the line with $\zeta_1=\Lambda_1^2$ and $\zeta_2$ unrestricted. Combining these results with 
Lemma~\ref{Lem:2cycles}, Lemma~\ref{Lem:SW}, and Remark~\ref{Rem:integrals}, we obtain
\begin{equation}
\begin{split}
 \oiint_{\Sigma^{(i+k-1,j+r+1-k)}_u}  dU\wedge \frac{dX}{Y}  
 =  - \, C^{(j)}_{2r} \,  \int_{\Lambda_2^2}^1 \frac{d\zeta_2}{\eta_2} \;  C^{(i)}_{2r} \, \int_0^{\Lambda_1^2} \frac{d\zeta_1}{\eta_1} 
  =  \oint_{\mathfrak{b}_j} \frac{d\zeta_2}{\eta_2}  \cdot   \oint_{\mathfrak{a}_i}  \frac{d\zeta_1}{\eta_1},
 \end{split}
\end{equation}
where we have used $C^{(j+r+1-k)}_{2r} \, C^{(i+k-1)}_{2r} = - \, C^{(i)}_{2r} \, C^{(j)}_{2r}$. The lemma follows from
$$
\oiint_{ \sigma_{i,j}} \omega 
  =  \iint_{ \gamma_i \times \gamma_j} \frac{d\zeta_1}{\eta_1} \boxtimes \frac{d\zeta_2}{\eta_2}
  =  \oiint_{ \mathfrak{a}_i \times \mathfrak{b}_j} \frac{d\zeta_1}{\eta_1} \boxtimes \frac{d\zeta_2}{\eta_2}   \;.
 $$ 
\end{proof}
\begin{remark}
Equation~(\ref{period_eqn}) simply states that the periods of a fixed holomorphic two-form
on the generalized Kummer surface must be the same when evaluated using 
either the fibration from Section~\ref{TwistFibJ7} or the isotrivial fibration from Section~\ref{IsoFibJ4}.
\end{remark}
We have arrived at the following theorem:
\begin{theorem}
The Multivariate Clausen Identity~(\ref{F2periodc}) is equivalent to Equation~(\ref{period_eqn}).
\end{theorem}
\begin{proof}
The generalized Kummer surface $S_0^{(r,2r-p,2r-q)}$ was constructed from the two curves
\begin{equation}
\label{ReminderSEs2}
 \begin{split}
 SE(\Lambda_1^2)^{2r}_{r,2r-p,2r-q}:&  \quad \eta_1^{2r} = \zeta_1^{3r-p-q} \; (\zeta_1-1)^{p} \; (\zeta_1-\Lambda_1^2)^{p} \;,\\
 SE(\Lambda_2^2)^{2r}_{r,p,q}:& \quad \eta_2^{2r} = \zeta_2^{p+q-r} \; (\zeta_2-1)^{2r-p} \; (\zeta_2-\Lambda_2^2)^{2r-p}  \;.
 \end{split}
\end{equation}
With $\beta_1=q/(2r)$, $\beta_2=p/(2r)$ and $\alpha=\beta_1 + \beta_2 - \frac{1}{2}$, we found
\begin{equation}
\begin{split}
& \,\frac{\Gamma\left(\beta_1 + \frac{1}{2}\right)}{\Gamma\left(\alpha\right) \, \Gamma\left(1-\beta_2\right)}  \frac{1}{\Lambda_1^{2\,\beta_1-1}}  \int_0^{\Lambda_1^2} \frac{d\zeta_1}{\eta_1}
= \hpg21{\beta_1 + \beta_2 -\frac{1}{2},\;\beta_2}{\beta_1 + \frac{1}{2}}{\Lambda_1^2} \;.
\end{split}
\end{equation}
Similarly, using $\zeta_2 = 1 - (1-\Lambda_2^2) \, \tilde{\zeta}_2$ we derived
\begin{equation}
\label{hgf1}
\begin{split}
& \,\frac{\Gamma\left(2\beta_2\right)}{\Gamma\left(\beta_2\right)^2}  \frac{(-1)^{1-\beta_2}}{\left(1-\Lambda_2^2\right)^{2 \, \beta_2 -1}}  \int_{\Lambda_2^2}^1 \frac{d\zeta_2}{\eta_2}
= \hpg21{\beta_1 + \beta_2 -\frac{1}{2},\;\beta_2}{2  \beta_2}{1-\Lambda_2^2} \;.
\end{split}
\end{equation}
Moreover, when written out explicitly Equation~(\ref{period_eqn}) is equivalent to
\begin{equation}
\label{hgf2}
\begin{split}
& \;  \frac{(-1)^{-\beta_2} }{C^{(j+r+1-k)}_{2r} \, C^{(i+k-1)}_{2r}}  \,  \frac{ \Gamma\left(\beta_1 + \frac{1}{2}\right) \, \Gamma(2\, \beta_2)}{\Gamma\left(\alpha\right) \, \Gamma\left(1-\beta_2\right) \, \Gamma(\beta_2)^2} \;
  \, (A-B)^{1-2 \, \beta_2} \, \oiint_{\Sigma^{(i+k-1,j+r+1-k)}_u}  du\wedge \frac{dx}{y}  \\
= \; &  \; \left( \Lambda_1 \, \Lambda_2\right)^\alpha 
 \left( \frac{\Gamma\left(\beta_1 + \frac{1}{2}\right)}{\Gamma\left(\alpha\right) \, \Gamma\left(1-\beta_2\right)}  \frac{1}{\Lambda_1^{2\,\beta_1-1}}  \int_0^{\Lambda_1^2} \frac{d\zeta_1}{\eta_1}\right) \, 
 \cdot \left(\frac{\Gamma\left(2\beta_2\right)}{\Gamma\left(\beta_2\right)^2}  \frac{(-1)^{1-\beta_2}}{\left(1-\Lambda_2^2\right)^{2 \, \beta_2 -1}}  \int_{\Lambda_2^2}^1 \frac{d\zeta_2}{\eta_2}\right).
 \end{split}
\end{equation}
where we have used Equation~(\ref{Relation_2forms}) and
\begin{equation}
\begin{split}
  \frac{2^{2-\frac{2p}{r}}  \, \Lambda_1^{\frac{3}{2}-\frac{p}{2r}-\frac{q}{2r}} \,  \Lambda_2^{\frac{1}{2}-\frac{p}{2r}+\frac{q}{2r}}}{(\Lambda_1^2-1)^{1-\frac{p}{r}}}  \, (A-B)^{1-2 \, \beta_2} 
  = \left( \Lambda_1 \, \Lambda_2\right)^\alpha \,   \frac{1}{\Lambda_1^{2\,\beta_1-1}}  \, \frac{1}{\left(1-\Lambda_2^2\right)^{2 \, \beta_2 -1}}  \;.
 \end{split}
\end{equation}
We simplify the left hand side using the result of Lemma~\ref{Lem:F2period} and the right hand side using Equations~(\ref{hgf1}) and (\ref{hgf2}). We obtain
\begin{equation}
\begin{split}
& \qquad \qquad \quad \frac{1}{(4 \,A)^{\alpha}} \,  \app2{\alpha;\;\beta_1,\; \beta_2}{2 \beta_1, \; 2\beta_2}{\frac{1}{A}, \; 1-\frac{B}{A}}\\
= \; & \, \big(\Lambda_1 \, \Lambda_2\big)^{\alpha} \; 
\hpg21{\beta_1 + \beta_2 -\frac{1}{2},\;\beta_2}{\beta_1 + \frac{1}{2}}{\Lambda_1^2}  \; \hpg21{\beta_1 + \beta_2 -\frac{1}{2},\;\beta_2}{2 \beta_2}{1-\Lambda_2^2}   \;.
\end{split}
\end{equation}
This is Equation~(\ref{F2periodb}) which is equivalent to the Multivariate Clausen Identity~(\ref{F2periodc}) and thus implies Theorem~\ref{thm1}.
\end{proof}

\bibliography{CDM15} 
\bibliographystyle{amsplain}

\end{document}